\documentclass[a4paper]{amsart}
\usepackage{amssymb,stmaryrd}

\newtheorem*{thma}{Theorem A}
\newtheorem*{thmb}{Theorem B}
\newtheorem*{thmc}{Theorem C}
\newtheorem*{thmd}{Theorem D}
\newtheorem*{thme}{Thereom E}
\newtheorem{thm}{Theorem}[section]
\newtheorem{lemma}[thm]{Lemma}
\newtheorem{cor}[thm]{Corollary}
\newtheorem{claim}{Claim}[thm]

\newtheorem{prop}[thm]{Proposition}
\newtheorem{fact}[thm]{Fact}

\theoremstyle{definition}
\newtheorem{defn}[thm]{Definition}
\newtheorem{conv}[thm]{Convention}
\newtheorem{question}[thm]{Question}

\theoremstyle{remark}
\newtheorem{remark}[thm]{Remark}

\DeclareMathOperator{\op}{op}
\DeclareMathOperator{\fp}{FP}
\DeclareMathOperator{\cub}{CUB}
\DeclareMathOperator{\add}{Add}
\DeclareMathOperator{\clps}{clps}
\DeclareMathOperator{\supp}{supp}
\DeclareMathOperator{\ord}{OR}
\DeclareMathOperator{\id}{id}
\DeclareMathOperator{\cf}{cf}
\DeclareMathOperator{\otp}{otp}
\DeclareMathOperator{\acc}{acc}
\DeclareMathOperator{\nacc}{nacc}
\DeclareMathOperator{\tr}{Tr}
\DeclareMathOperator{\dom}{dom}
\DeclareMathOperator{\im}{Im}
\DeclareMathOperator{\reg}{Reg}
\DeclareMathOperator{\sing}{Sing}
\DeclareMathOperator{\refl}{Refl}
\DeclareMathOperator{\cof}{cof}

\DeclareMathOperator{\pr}{pr}

\DeclareMathOperator{\Col}{Col}
\DeclareMathOperator{\ns}{NS}
\DeclareMathOperator{\dl}{Dl}
\DeclareMathOperator{\cl}{cl}

\newcommand{\forces}{\Vdash}
\newcommand{\s}{\subseteq}
\newcommand{\br}{\blacktriangleright}
\newcommand*{\axiomfont}[1]{\textsf{\textup{#1}}}
\newcommand{\zf}{\axiomfont{ZF}}
\newcommand{\zfc}{\axiomfont{ZFC}}
\newcommand{\gch}{\axiomfont{GCH}}

\newcommand{\mm}{\axiomfont{MM}}
\newcommand{\lcc}{\axiomfont{LCC}}
\renewcommand{\mid}{\mathrel{|}\allowbreak}
\renewcommand{\restriction}{\mathbin\upharpoonright}

\newcommand{\symdiff}{\mathbin\triangle}

\newcommand{\redup}{\mathrel{\hookrightarrow_{\mathbf p}}}
\newcommand{\redum}{\mathrel{\hookrightarrow_{BM}}}
\newcommand{\redub}{\mathrel{\hookrightarrow_B}}
\newcommand{\reduc}{\mathrel{\hookrightarrow_c}}
\newcommand{\reduO}{\mathrel{\hookrightarrow_1}}
\newcommand{\nredum}{\mathrel{\not\mskip-\thinmuskip\hookrightarrow_{BM}}}
\newcommand{\nredub}{\mathrel{\not\mskip-\thinmuskip\hookrightarrow_B}}

\newcommand*{\sq}[1]{\mathrel{\le^{#1}}}
\newcommand*{\sqc}[1]{\mathrel{\subseteq^{#1}}}
\providecommand{\myceil}[1]{{}^\ulcorner #1{}^\urcorner }

\title{Fake reflection}

\author{Gabriel Fernandes}
\address{Department of Mathematics, Bar-Ilan University, Ramat-Gan 5290002, Israel.}
\urladdr{http://u.math.biu.ac.il/\textasciitilde zanettg}

\author{Miguel Moreno}
\address{Department of Mathematics, Bar-Ilan University, Ramat-Gan 5290002, Israel.}
\urladdr{http://u.math.biu.ac.il/\textasciitilde morenom3}

\author{Assaf Rinot}
\address{Department of Mathematics, Bar-Ilan University, Ramat-Gan 5290002, Israel.}
\urladdr{http://www.assafrinot.com}


\begin{document}
\begin{abstract} We introduce a generalization of stationary set reflection
which we call \emph{filter reflection}, and show it is compatible with the axiom of constructibility
as well as with strong forcing axioms. We prove the independence of filter reflection from ZFC,
and present applications of filter reflection to the study of canonical equivalence relations of the higher Cantor and Baire spaces.
\end{abstract}

\maketitle
\section{Introduction}

Throughout the paper, $\kappa$ denotes an uncountable cardinal satisfying $\kappa^{<\kappa}=\kappa$.

Motivated by questions arising from the study of the higher Cantor and Baire spaces (also known as the generalized Cantor and Baire spaces), $2^\kappa$ and $\kappa^\kappa$,
respectively, in this paper, we introduce and study a new notion of reflection which we call \emph{filter reflection}.

\begin{defn}\label{Def0} Suppose $X$ and $S$ are stationary subsets of $\kappa$,
and $\vec{\mathcal F}=\langle\mathcal F_\alpha\mid\alpha\in S\rangle$ is a sequence
such that, for each $\alpha\in S$, $\mathcal F_\alpha$ is a filter over $\alpha$.
\begin{enumerate}
\item We say that $\vec{\mathcal F}$  \emph{captures clubs} iff, for every club $C\s\kappa$,
the set $\{\alpha\in S\mid C\cap\alpha\notin\mathcal F_\alpha\}$ is non-stationary;
\item We say that \emph{$X$ $\vec{\mathcal F}$-reflects to $S$} iff $\vec{\mathcal F}$ captures clubs and,
for every stationary $Y\s X$, the set $\{\alpha\in S\mid Y\cap\alpha\in\mathcal F_\alpha^+\}$ is stationary;
\item We say that \emph{$X$ $\mathfrak f$-reflects to $S$} iff there exists a sequence of filters $\vec{\mathcal F}$
over a stationary subset $S'$ of $S$ such that $X$ $\vec{\mathcal F}$-reflects to $S'$.
\end{enumerate}
\end{defn}

It is not hard to verify (see Lemma~\ref{lemma53} below)
that if $X$ $\mathfrak f$-reflects to $S$, then the equivalence relation $=_X$ on the space $\kappa^\kappa$ (see Definition~\ref{canonicalER} below)
admits a Lipschitz reduction to the equivalence relation $=_S$.
A detailed study of the effect of filter reflection (and stronger forms of which) on the higher Baire and Cantor spaces 
is given in Section~\ref{section2} below,
but as we feel that the notion of filter reflection is of independent interest, let us take a moment to inspect Definition~\ref{Def0}.
\begin{itemize}
\item A sequence $\vec{\mathcal F}$ over a stationary $S\s \cof({>}\omega)$ that captures clubs can be obtained in \zfc;
for every $\alpha\in S$, simply take $\mathcal F_\alpha$ to be the club filter over $\alpha$, $\cub(\alpha)$.

An alternative way to consistently get a sequence $\vec{\mathcal F}$ that captures clubs is by consulting a $\diamondsuit^*_S$-sequence,
in which case, each filter $\mathcal{F}_\alpha$ will have a small base.

\item If we omit the requirement that $\vec{\mathcal F}$ captures clubs from Definition~\ref{Def0}(2),
then again such a sequence may be constructed in $\zfc$, with each $\mathcal F_\alpha$ being a principal ultrafilter.
However, if $S$ is an ineffable set (see Definition \ref{pointer-ineffable} below), then capturing implies that $\mathcal F_\alpha\supseteq\cub(\alpha)$ for most $\alpha\in S$.

An alternative way to consistently get ``reflection minus capturing'' goes through forcing axioms;
for instance, it follows from the main result of \cite{MR2224055}
that \textsf{MRP} (a principle weaker than the Proper Forcing Axiom)
entails that for $X:=\omega_2\cap\cof(\omega)$ and $S:=\omega_2\cap\cof(\omega_1)$,
there exists a sequence $\vec{\mathcal F}$ with $\mathcal F_\alpha\s \cub(\alpha)$ for all $\alpha\in S$,
and satisfying that for every stationary $Y\s X$, the set $\{\alpha\in S\mid Y\cap\alpha\in\mathcal F_\alpha^+\}$ is stationary.
\end{itemize}

Thus, filter reflection is the conjunction of two requirements, each being a consequence of \zfc,
and the challenge is in simultaneously satisfying both.
The special case in which $\mathcal F_\alpha=\cub(\alpha)$ for all $\alpha\in S'$
is better known as the assertion that ``every stationary subset of $X$ reflects in $S$''
that goes back to some of the milestone papers in the history of forcing with large cardinals \cite{MR0416925,magidor1982reflecting,MR783595,MR1029909}.

In this paper, we pay special attention to the case in which $\mathcal F_\alpha\nsupseteq\cub(\alpha)$ for all $\alpha\in S'$;
we regard this special case by the name \emph{fake reflection}.
The obvious advantage of fake reflection over genuine reflection is that
the former makes sense even for stationary sets $S$ consisting of points of countable cofinality.
The other --- somewhat unexpected --- advantage is that its consistency does not require large cardinals:
\begin{thma} For all stationary subsets $X$ and $S$ of $\kappa$,
there exists a ${<}\kappa$-closed $\kappa^+$-cc forcing extension,
in which  $X$ $\mathfrak f$-reflects to $S$.
\end{thma}

We shall also show that (a strong form of) fake reflection is a consequence of two axioms of contradictory nature,
that is, the axiom of constructibility (${\sf V=L}$) and Martin's Maximum (${\sf MM}$).

After realizing that fake reflection is so prevalent,
we tried for a while to prove that it is a consequence of \zfc,
or, at least, a consequence of $\diamondsuit^+_S$,
but in vain. Eventually, we discovered that this is not the case:\footnote{See also Corollary~\ref{cor525} below.}
\begin{thmb} There exists a cofinality-preserving forcing extension
in which for all two disjoint stationary subsets $X,S$ of $\kappa$,
$X$ does not $\mathfrak f$-reflect to $S$.
\end{thmb}

Theorems $A$ and $B$ give two extreme configurations of filter reflection.
The next theorem allows considerably more subtle configurations,
as it gives rise to models in which filter reflection is no more general than the classic notion of reflection:
\begin{thmc} If $\kappa$ is strongly inaccessible (e.g., a Laver-indestructible large cardinal),
then, in the forcing extension by $\add(\kappa,\kappa^+)$, for all two disjoint stationary subsets $X,S$ of $\kappa$,
the following are equivalent:
\begin{itemize}
\item $X$ $\mathfrak f$-reflects to $S$;
\item every stationary subset of $X$ reflects in $S$.
\end{itemize}
\end{thmc}

As hinted earlier, filter reflection has a strong connection with the Borel-reducibility hierarchy of the higher Cantor and Baire spaces. To exemplify:\footnote{See Section~\ref{section2} for missing definitions.}
\begin{thmd} For all two disjoint stationary subsets $X,S$ of $\kappa$. If $X$ $\mathfrak f$-reflects to $S$ and $S$ $\mathfrak f$-reflects to $X$, then there is a  map $f:\kappa^\kappa\rightarrow\kappa^\kappa$ simultaneously witnessing $${=_X}\reduO{=_S}\ \&\ {=_S}\reduO{=_X}.$$
\end{thmd}

Filter reflection also provides us with tools to answer a few questions from the literature. For instance, the following answers Question~2.12 of \cite{AHKM} for the case $\kappa=\omega_2$:
\begin{thme}
The consistency of $\mm$ implies the consistency of
$${=_{\omega_2\cap\cof (\omega)}}\reduO{=^2_{\omega_2\cap\cof (\omega_1)}}\ \&\  {=^2_{\omega_2\cap\cof (\omega_1)}}\nredum{=_{\omega_2\cap\cof (\omega)}}.$$
\end{thme}

\subsection{Organization of this paper} In Section~\ref{section2},
we introduce the notion of filter reflection and its variations and study their effect on canonical equivalence relations over the higher Baire and Cantor spaces.
The proof of Theorem~D will be found there.

In Section~\ref{Section3}, we study
strong and simultaneous forms of filter reflection, proving various sufficient and equivalent conditions for these principles to hold.

In Section~\ref{section4}, we present a poset that forces filter and fake reflection to hold,
thus, proving Theorem~A.

In Section~\ref{sectionFFR}, we present a poset that forces filter and fake reflection to fail,
along the way, proving Theorems B and C.

In Section~\ref{nonreductionsection}, we extend the analysis of Section~\ref{sectionFFR} to get strong failures of Baire measurable reductions.
Along the way, we answer a few questions from the literature, and prove Theorem~E.

\subsection{Notation and conventions}
For sets $X,Y$, their difference $\{ x\in X\mid x\notin Y\}$
is denoted by $X\setminus Y$, whereas the notation $X-Y$ is reserved for
the relative of Shelah's approachability ideal that we introduce in Section~\ref{sectionFFR}.

By a \emph{filter} we always mean a proper filter, so that is does not contain the empty set.
For a filter $\mathcal F$ over a set $X$, we let $\mathcal F^+:=\{Y\in\mathcal P(X)\mid X\setminus Y\notin\mathcal F\}$.
For an ideal $\mathcal I$ over a set $X$ and a subset $Y\s X$, we write $\mathcal I\restriction Y$ for $\mathcal I\cap\mathcal P(Y)$.

The class of ordinals is denoted by $\ord$.
The class of ordinals of cofinality $\mu$ is denoted by $\cof(\mu)$, and
the class of ordinals of cofinality greater than $\mu$ is denoted by $\cof({>}\mu)$.
The class of infinite singular ordinals is denoted by $\sing$.
The class of infinite regular cardinals is denoted by $\reg$.
We write $\sing(\kappa)$ for ${\sing}\cap\kappa$,
and $\reg(\kappa)$ for ${\reg}\cap\kappa$.
For $m<n<\omega$, we denote $S^n_m:=\aleph_n\cap\cof(\aleph_m)$.

For a set of ordinals $a$, we write $\cl(a):=\{ \sup(a\cap\alpha)\mid \alpha\in\ord, a\cap\alpha\neq\emptyset\}$,
$\acc(a):=\{\alpha\in a\mid \sup (a\cap\alpha)=\alpha>0\}$,
$\acc^+(a):=\{\alpha<\sup(a)\mid \sup(a\cap\alpha)=\alpha>0\}$ 
and $\nacc(a):=a\setminus \acc(a)$.
For a stationary $Y\s\kappa$, we write $\tr(Y):=\{\alpha\in \kappa\cap\cof({>}\omega)\mid Y\cap\alpha\text{ is stationary in }\alpha\}$.

We let $\Col(\lambda,{<}\kappa)$ denote L\'evy's notion of forcing for collapsing $\kappa$ to $\lambda^+$,
and let $\add(\kappa,\theta)$ denote Cohen's notion of forcing for adding $\theta$ many functions from $\kappa$ to $2$.
For all $\alpha<\kappa$, $p:\alpha\rightarrow2$, and $i<2$, we denote by
$p{}^\curvearrowright i$ the unique function $p'$ extending $p$ satisfying $\dom(p')=\alpha+1$ and $p'(\alpha)=i$.

\section{Filter reflection and Lipschitz reductions}\label{section2}

Throughout this section, we let $S,X$ denote stationary subsets of $\kappa$,
and consider sequences $\vec{\mathcal F}=\langle\mathcal F_\alpha\mid\alpha\in S\rangle$,
where, for all $\alpha\in S$, $\mathcal F_\alpha$ is a filter over $\alpha$.
Recall that $\vec{\mathcal F}$ is said to \emph{captures clubs} iff, for every club $C\s\kappa$,
the set $\{\alpha\in S\mid C\cap\alpha\notin\mathcal F_\alpha\}$ is non-stationary.

\begin{defn}[Filter reflection]\label{fake-ref}\hfill
\begin{enumerate}
\item We say that \emph{$X$ $\vec{\mathcal F}$-reflects to $S$} iff $\vec{\mathcal F}$ captures clubs and,
for every stationary $Y\s X$, the set $\{\alpha\in S\mid Y\cap\alpha\in\mathcal F_\alpha^+\}$ is stationary;
\item We say that \emph{$X$ strongly $\vec{\mathcal F}$-reflects to $S$} iff $\vec{\mathcal F}$ captures clubs and,
for every stationary $Y\s X$, the set $\{\alpha\in S\mid Y\cap\alpha\in\mathcal F_\alpha\}$ is stationary;
\item We say that \emph{$X$ $\vec{\mathcal F}$-reflects  with $\diamondsuit$ to $S$} iff
$\vec{\mathcal F}$ captures clubs and there exists a sequence $\langle Y_\alpha\mid\alpha\in S\rangle$ such that,
for every stationary $Y\s X$, the set $\{\alpha\in S\mid Y_\alpha=Y\cap\alpha\ \&\ Y\cap\alpha\in\mathcal F_\alpha^+\}$ is stationary.
\end{enumerate}
\end{defn}
\begin{defn}
We say that \emph{$X$ $\mathfrak f$-reflects to $S$} whenever there exists a stationary $S'\s S$ and sequence of filters $\vec{\mathcal F}=\langle\mathcal F_\alpha\mid\alpha\in S'\rangle$ such that $X$ $\vec{\mathcal F}$-reflects to $S'$.

The same convention applies to strong $\mathfrak f$-reflection and to $\mathfrak f$-reflection with $\diamondsuit$.
\end{defn}

Notice that a priori there is no reason to believe that ``$X$ $\mathfrak f$-reflects to $S$'' implies the existence of a sequence $\vec{\mathcal F}$ for which ``$X$ $\vec{\mathcal F}$-reflects to $S$''.
To put our finger at the challenge, it is in the requirement of capturing clubs.

\begin{prop}\label{prop23} For stationary subsets $X$ and $S$  of $\kappa$, $(1)\implies(2)\implies(3)$:
\begin{enumerate}
\item $X$ $\mathfrak f$-reflects with $\diamondsuit$ to $S$;
\item $X$ strongly $\mathfrak f$-reflects to $S$;
\item $X$ $\mathfrak f$-reflects to $S$.
\end{enumerate}
\end{prop}
\begin{proof} $(1)\implies(2)$ Suppose $\vec{\mathcal F}=\langle\mathcal F_\alpha\mid\alpha\in S'\rangle$ and $\langle Y_\alpha\mid\alpha\in S'\rangle$
witness together that  $X$ $\mathfrak f$-reflects with $\diamondsuit$ to $S$.
Let $S'':=\{ \alpha\in S'\mid Y_\alpha\in\mathcal F_\alpha^+\}$.
For each $\alpha\in S''$, let $\bar{\mathcal F}_\alpha$ be the filter over $\alpha$ generated by $\mathcal{F}_\alpha\cup\{Y_\alpha\}$.
Evidently, $\langle \bar{\mathcal F}_\alpha\mid \alpha\in S''\rangle$ witnesses that $X$ strongly $\mathfrak f$-reflects to $S$.
\end{proof}

The following is obvious:
\begin{lemma}[Monotonicity]\label{monotonicity} For stationary sets $Y\s X\s\kappa$ and $S\s T\s\kappa$:
\begin{enumerate}
\item If $X$ $\mathfrak f$-reflects to $S$, then $Y$ $\mathfrak f$-reflects to $T$;
\item If $X$ strongly $\mathfrak f$-reflects to $S$, then $Y$ strongly $\mathfrak f$-reflects to $T$;
\item If $X$ $\mathfrak f$-reflects with $\diamondsuit$ to $S$, then
$Y$ $\mathfrak f$-reflects with $\diamondsuit$ to $T$.\qed
\end{enumerate}
\end{lemma}

\begin{defn}\label{canonicalER}
\begin{enumerate}
\item For all $\eta,\xi\in\kappa^\kappa$, denote
$$\Delta(\eta,\xi):=\min(\{\alpha<\kappa\mid \eta(\alpha)\neq\xi(\alpha)\}\cup\{\kappa\}).$$
\item For every $\theta\in[2,\kappa]$, the equivalence relation $=^\theta_S$ over $\theta^\kappa$
is defined via
$$\eta\mathrel{=^\theta_S}\xi\text{ iff }\{\alpha\in S\mid \eta(\alpha)\neq\xi(\alpha)\}\text{ is non-stationary}.$$
\item The special case $=^\kappa_S$ is denoted by $=_S$.
\end{enumerate}
\end{defn}

The above equivalence relation is an important building block in the study of the connection between model theory and generalized descriptive set theory (cf.~\cite[Chapter~6]{FHK} and \cite[Theorem~7]{HKM}).

\begin{defn}\label{defreduction}
For $i<2$, let $X_i$ be some space from the collection $\{\theta^\kappa\mid \theta\in[2,\kappa]\}$.

Let $R_0$ and $R_1$ be binary relations over $X_0$ and $X_1$, respectively.
A function $f:X_0\rightarrow X_1$ is said to be:
\begin{itemize}
\item[(a)] a \emph{reduction of $R_0$ to $R_1$} iff, for all $\eta,\xi\in X_0$, $$\eta\mathrel{R_0}\xi\text{ iff }f(\eta)\mathrel{R_1}f(\xi).$$
\item[(b)] $1$-\emph{Lipschitz} iff for all $\eta,\xi\in X_0$,  $$\Delta(\eta,\xi)\le\Delta(f(\eta),f(\xi)).$$
\end{itemize}
\end{defn}
The existence of a function $f$ satisfying (a) and (b) is  denoted by ${R_0}\reduO{R_1}$.
We likewise define ${R_0}\reduc{R_1}$, ${R_0}\redub{R_1}$ and  ${R_0}\redum{R_1}$ once we replace clause~(b) by a continuous, Borel, or Baire measurable map, respectively.\footnote{See Section~\ref{nonreductionsection} for definitions and topological properties of the higher Cantor and Baire spaces.}
The following is obvious.
\begin{lemma}[Monotonicity]\label{transitivity} For
\begin{itemize}
\item $\mathbf p\in\{1,c,B,BM\}$,
\item stationary sets $X\s X'\s\kappa$ and $Y\s Y'\s\kappa$, and
\item ordinals $2\le\theta\le\theta'\le\kappa$ and $2\le\lambda\le\lambda'\le\kappa$,
\end{itemize}
${=^{\theta'}_{X'}}\redup{=^\lambda_Y}$ entails ${=^\theta_{X}}\redup{=^{\lambda'}_{Y'}}$.\qed
\end{lemma}

\begin{lemma}\label{lemma53} If $X$ $\mathfrak f$-reflects to $S$, then ${=_X}\reduO{=_S}$.
\end{lemma}
\begin{proof} Suppose that $\vec{\mathcal F}=\langle \mathcal F_\alpha\mid\alpha\in S'\rangle$
witnesses that $X$ $\mathfrak f$-reflects to $S$.
For every $\alpha\in S'$,
define an equivalence relation $\sim_\alpha$ over $\kappa^\alpha$ by letting
$\eta\sim_\alpha\xi$ iff there is $W\in \mathcal F_\alpha$ such that $W\cap X\subseteq\{\beta<\alpha\mid \eta(\beta)=\xi(\beta)\}$.
As there are at most $|\kappa^\alpha|$ many equivalence classes and as $\kappa^{<\kappa}=\kappa$,
we may attach to each equivalence class $[\eta]_{\sim_\alpha}$ a unique ordinal (a \emph{code}) in $\kappa$,
which we shall denote by $\myceil{[\eta]_{\sim_\alpha}}$.
Next, define a map $f:\kappa^\kappa\rightarrow\kappa^\kappa$ by letting
for all $\eta\in\kappa^\kappa$ and $\alpha<\kappa$:
$$f(\eta)(\alpha):=\begin{cases}
\myceil{[\eta\restriction\alpha]_{\sim_\alpha}},&\text{if }\alpha\in S';\\
0,&\text{otherwise}.\end{cases}$$

As, for all $\eta\in \kappa^\kappa$ and $\alpha<\kappa$, $f(\eta)(\alpha)$ depends only on $\eta\restriction\alpha$, $f$ is $1$-Lipschitz.
To see that it forms a reduction from $=_X$ to $=_S$, let $\eta,\xi$ be arbitrary elements of $\kappa^\kappa$. There are two main cases to consider:

$\br$ If $\eta=_X\xi$, then let us fix a club $C$ such that $C\cap X\subseteq \{\beta<\kappa\mid \eta(\beta)=\xi(\beta)\}$. 
Since $\vec{\mathcal F}$ captures clubs, let us fix a club $D\s\kappa$ such that, for all $\alpha\in D\cap S'$, $C\cap\alpha\in\mathcal F_\alpha$.
We claim that $D$ is disjoint from $\{\alpha\in S\mid f(\eta)(\alpha)\neq f(\xi)(\alpha)\}$,
so that indeed $f(\eta)=_Sf(\xi)$. To see this, let $\alpha\in D$ be arbitrary.

$\br\br$ If $\alpha\not\in S'$, then $f(\eta)(\alpha)=0=f(\xi)(\alpha)$.

$\br\br$ If $\alpha\in S'$, then for $W:=C\cap\alpha$, we have that $W\in\mathcal F_\alpha$
and $W\cap X\subseteq\{\beta<\alpha\mid \eta(\beta)=\xi(\beta)\}$,
so that $[\eta\restriction\alpha]_{\sim_\alpha}=[\xi\restriction\alpha]_{\sim_\alpha}$
and $f(\eta)(\alpha)=f(\xi)(\alpha)$.

$\br$ If $\eta\neq_X\xi$, then $Y:=\{\beta\in X\mid \eta(\beta)\neq\xi(\beta)\}$ is stationary.
In effect, $T:=\{\alpha\in S'\mid Y\cap\alpha\in\mathcal F_\alpha^+\}$ is stationary.
Now, for every $\alpha\in T$,
it must be the case that any $W\in \mathcal F_\alpha$ meets $Y$,
so that $W\cap X\nsubseteq\{\beta<\alpha\mid \eta(\beta)=\xi(\beta)\}$,
and $[\eta\restriction\alpha]_{\sim_\alpha}\neq[\xi\restriction\alpha]_{\sim_\alpha}$.
It follows that $\{\alpha\in S'\mid f(\eta)(\alpha)\neq f(\xi)(\alpha)\}$ covers the stationary set $T$,
so that $f(\eta)\neq_Sf(\xi)$.
\end{proof}

\begin{lemma}\label{redstrongf} If $X$ strongly $\mathfrak f$-reflects to $S$, then, for every $\theta\in[2,\kappa]$, ${=^\theta_X}\reduO{=^\theta_S}$.
\end{lemma}
\begin{proof} By the preceding lemma and by the implication $(2)\implies(3)$ of Proposition~\ref{prop23}, we may assume that $\theta\in[2,\kappa)$.
Suppose $\vec{\mathcal F}=\langle \mathcal F_\alpha\mid\alpha\in S'\rangle$ is a sequence witnessing that $X$ strongly $\mathfrak f$-reflects to $S$.
Define a map $f:\theta^\kappa\rightarrow\theta^\kappa$ as follows.
For every $\alpha\in S'$ and $\eta\in\theta^\kappa$,
if there exists $W\in \mathcal F_\alpha$ and $i<\theta$ such that $W\cap X\subseteq \{\beta<\alpha\mid\eta(\beta)=i\}$,
then it is unique (since $\mathcal F_\alpha$ is a filter), and so we let $f(\eta)(\alpha):=i$.
If there is no such $i$ or if $\alpha\notin S'$, then we simply let $f(\eta)(\alpha):=0$.

As, for all $\eta\in \kappa^\kappa$ and $\alpha<\kappa$, $f(\eta)(\alpha)$ depends only on $\eta\restriction\alpha$, $f$ is $1$-Lipschitz.
To see that it forms a reduction, let $\eta,\xi$ be arbitrary elements of $\kappa^\kappa$. There are two cases to consider:

$\br$ If $\eta=^\theta_X\xi$, then since $\vec{\mathcal F}$ captures clubs, we infer that $f(\eta)=^\theta_Sf(\xi)$,
very much like the proof of this case in Lemma~\ref{lemma53}.

$\br$ If $\eta\neq^\theta_X\xi$, then, as $\theta<\kappa$, we may find $i\neq j$ for which $Y:=\{\beta\in X\mid \eta(\beta)=i\ \&\ \xi(\beta)=j\}$ is stationary.
In effect, $T:=\{\alpha\in S'\mid Y\cap\alpha\in\mathcal F_\alpha\}$ is stationary.
Now, for every $\alpha\in T$,
we have that $f(\eta)(\alpha)=i$ while  $f(\xi)(\alpha)=j$.
It follows that $f(\eta)\neq_Sf(\xi)$.
\end{proof}

Given Lemmas \ref{lemma53} and \ref{redstrongf}, it is natural to ask whether ${=^2_X}\reduO{=^2_S}$ implies ${=_X}\reduO{=_S}$. The following provides a sufficient condition.

\begin{lemma} If $\kappa$ is not a strong limit,
then ${=^2_X}\reduO{=^2_S}$ implies ${=_X}\reduO{=_S}$.
\end{lemma}
\begin{proof} Suppose that $\kappa$ is not a strong limit.
As we are working under the hypothesis that $\kappa^{<\kappa}=\kappa$,
this altogether means that we may fix a bijection $h:\kappa\leftrightarrow 2^\lambda$ for some cardinal $\lambda<\kappa$.
In effect, there is a bijection $\pi:\kappa^\kappa\leftrightarrow (2^\kappa)^\lambda$
satisfying that, for every $\eta\in\kappa^\kappa$,
$\pi(\eta)=\langle \eta_i\mid i<\lambda\rangle$  iff, for all $i<\lambda$ and $\alpha<\kappa$:
$$\eta_i(\alpha)=h(\eta(\alpha))(i).$$

Suppose that $f:2^\kappa\rightarrow2^\kappa$ is a continuous reduction from $=^2_X$ to $=^2_S$.
Define a map $g:\kappa^\kappa\rightarrow\kappa^\kappa$ via $g(\eta):=\zeta$ iff
$\pi(\zeta)=\langle f(\eta_i)\mid i<\lambda\rangle$.
\begin{claim} $g$ is $1$-Lipschitz.
\end{claim}
\begin{proof} Notice that $\pi$ is $1$-Lipschitz and $g$ is equivalent to the function $\pi^{-1}\circ f\circ\pi$. The claim follows from the fact that $f$ is $1$-Lipschitz.
\end{proof}
To see that $g$ is a reduction from $=_X$ to $=_S$, let $\eta,\xi\in\kappa^\kappa$ be arbitrary.
Denote $\langle \eta_i\mid i<\lambda\rangle:=\pi(\eta)$
and $\langle \xi_i\mid i<\lambda\rangle:=\pi(\xi)$.

$\br$ If $\eta=_X\xi$, then let us fix a club $C\s\kappa$ such that $C\cap X\s\{\alpha<\kappa\mid \eta(\alpha)=\xi(\alpha)\}$.
It easy to see that for every $i<\lambda$,
we have $C\cap X\s\{\alpha<\kappa\mid \eta_i(\alpha)=\xi_i(\alpha)\}$,
so that ${\eta_i}=_X^2{\xi_i}$ and we may pick a club $D_i\s\kappa$
such that $D_i\cap S\s\{\alpha<\kappa\mid f(\eta_i)(\alpha)=f(\xi_i)(\alpha)\}$.
Then $D:=\bigcap_{i<\lambda}D_i$ is a club,
and we have $D\cap S\s\{\alpha<\kappa\mid g(\eta)(\alpha)=g(\xi)(\alpha)\}$,
so that ${g(\eta)}=_S{g(\xi)}$.

$\br$ If $\neg(\eta=_X\xi)$, then $Y:=\{\alpha\in X\mid \eta(\alpha)\neq\xi(\alpha)\}$ is stationary.
Since $\pi$ is a bijection, for every $\alpha\in Y$, there is some $i<\lambda$ such that $\eta_i(\alpha)\neq\xi_i(\alpha)$.
It follows that there exists some $i<\lambda$ such that $Y_i:=\{\alpha\in X\mid \eta_i(\alpha)\neq\xi_i(\alpha)\}$ is stationary. Consequently,
the set $T:=\{\alpha\in S\mid f(\eta_i)(\alpha)\neq f(\xi_i)(\alpha)\}$ is stationary.
In particular, the set $\{\alpha\in S\mid g(\eta)(\alpha)\neq g(\xi)(\alpha)\}$
covers $T$, so that $\neg({g(\eta)}=_S{g(\xi)})$.
\end{proof}
\begin{remark} The same holds true once replacing $\reduO$ in the preceding by $\redup$ for any choice of $\mathbf p\in\{c,B,BM\}$.
\end{remark}

\begin{defn}
\begin{enumerate}
\item The quasi-order $\sq S$ over $\kappa^\kappa$
is defined via
$${\eta}\sq{S}{\xi}\text{  iff }\{\alpha\in S\mid \eta(\alpha)>\xi(\alpha)\}\text{ is non-stationary}.$$
\item The quasi-order $\sqc S$ over $2^\kappa$  is defined via:
$$\eta\sqc S\xi\text{ iff }\{\alpha\in S\mid \eta(\alpha)>\xi(\alpha)\}\text{ is non-stationary}.$$
\end{enumerate}
\end{defn}
\begin{remark} Note that  $\sqc S$ is nothing but ${\sq S}\cap(2^\kappa\times 2^\kappa)$.
\end{remark}

In \cite{FMR}, we addressed the problem of universality of $\sqc S$. The next theorem addresses a specific instance of this problem.

\begin{thm}\label{lemma43} If $X$ $\mathfrak f$-reflects with $\diamondsuit$ to $S$, then ${\sq X}\reduO{\sqc S}$.
\end{thm}
\begin{proof}
Let $\vec{\mathcal F}=\langle \mathcal F_\alpha\mid\alpha\in S'\rangle$ and $\langle Y_\alpha\mid\alpha\in S'\rangle$
witness together that $X$ $\mathfrak f$-reflects with $\diamondsuit$ to $S$.
Let $S'':=\{ \alpha\in S'\mid Y_\alpha\in\mathcal F_\alpha^+\}$.
For each $\alpha\in S''$, let $\bar{\mathcal F}_\alpha$ be the filter over $\alpha$ generated by $\mathcal{F}_\alpha\cup\{Y_\alpha\}$.

\begin{claim}\label{claim2131} There exists a sequence $\langle \eta_\alpha\mid\alpha\in S''\rangle$ such that,
for every stationary $Y\s X$ and every $\eta\in\kappa^\kappa$, the set $\{\alpha\in S''\mid \eta_\alpha=\eta\restriction\alpha\ \&\ Y\cap\alpha\in \bar{\mathcal F}_\alpha\}$ is stationary.
\end{claim}
\begin{proof}  Let $C:=\acc^+(X)$ and $B:=X\setminus C$,
so that $C$ is a club in $\kappa$ and $B$ is a non-stationary subset of $\kappa$ of cardinality $\kappa$.
As $|B|=\kappa^{<\kappa}$,\footnote{By convention, we always assume that $\kappa^{<\kappa}=\kappa$, but this time this also follows from the hypothesis that $X$ $\mathfrak f$-reflects with $\diamondsuit$ to $S$.}
 let us fix an injective enumeration  $\{ a_\beta\mid \beta\in B\}$ of the elements of $\kappa^{<\kappa}$.
Then, for each $\alpha\in S''$, let $\eta_\alpha:=(\bigcup\{ a_\beta\mid \beta\in Y_\alpha\cap B\})\cap(\alpha\times\alpha)$.

To see that $\langle \eta_\alpha\mid\alpha\in S''\rangle$ is as sought, fix an arbitrary $\eta\in\kappa^\kappa$ and an arbitrary stationary $Y\s X$.
Let $f:\kappa\rightarrow B$ be the unique function to satisfy that, for all $\epsilon<\kappa$, $a_{f(\epsilon)}=\eta\restriction\epsilon$.
Evidently, $Y\cap C$ is a stationary subset of $X$ disjoint from $\im(f)$.
In particular, $Y':=(Y\cap C)\uplus\im(f)$ is a stationary subset of $X$,
and hence $G:=\{\alpha\in S'\mid Y_\alpha=Y'\cap\alpha\ \&\ Y'\cap\alpha\in\mathcal F_\alpha^+\}$ is a stationary subset of $S''$.
Now, as $\vec{\mathcal F}$ captures clubs, let us fix a club $D\s\kappa$ such that, for all $\alpha\in D\cap S'$, $C\cap\alpha\in\mathcal F_\alpha$.
Consider $T:=\{\alpha\in G\cap D\mid f[\alpha]\s\alpha\ \&\ \eta[\alpha]\s\alpha\}$
which is a stationary subset of $S''$.
Let $\alpha\in T$ be arbitrary.

$\bullet$ As $\alpha\in D$, $C\cap\alpha\in\mathcal F_\alpha\s\bar{\mathcal F}_\alpha$,
and as $\alpha\in G$, $Y'\cap\alpha=Y_\alpha\in\bar{\mathcal F}_\alpha$.
Therefore, the intersection $Y'\cap C\cap\alpha$ is in $\bar{\mathcal F}_\alpha$.
But $Y'\cap C\cap\alpha=Y\cap C\cap\alpha$,
and hence the superset $Y\cap\alpha$ is in $\bar{\mathcal F}_\alpha$, as well.

$\bullet$ As $\alpha\in G$, $Y_\alpha=Y'\cap\alpha$, and we infer that $Y_\alpha\cap B=\im(f)\cap\alpha$.
As $f[\alpha]\s\alpha$, we get that $f[\alpha]\s Y_\alpha\cap B\s\im(f)$.
As $\eta[\alpha]\s\alpha$, we get that $\eta\restriction\alpha=\eta\cap(\alpha\times\alpha)$.
Recalling the definition of $f$ and the definition of $\eta_\alpha$,
it follows that $\eta\restriction\alpha\s \eta_\alpha\s\eta$,
so that $\eta_\alpha=\eta\restriction\alpha$.
\end{proof}

For every $\alpha\in S''$, define a quasi-order $\preceq_\alpha$ over $\kappa^\alpha$ by letting
$\eta\preceq_\alpha\xi$ iff there is $W\in\bar{\mathcal F}_\alpha$ such that $W\cap X\subseteq\{\beta<\alpha\mid \eta(\beta)\leq\xi(\beta)\}$.
Let $\langle \eta_\alpha\mid\alpha\in S''\rangle$ be given by the preceding claim.
Define a map $f:\kappa^\kappa\rightarrow 2^\kappa$ by letting
for all $\eta\in\kappa^\kappa$ and $\alpha<\kappa$:
$$    f(\eta)(\alpha) := \begin{cases}
              1,& \text{if } \alpha\in S''\ \&\  \eta_\alpha\mathrel{\preceq_\alpha}\eta\restriction\alpha;\\
               0,& \text{otherwise}.
                      \end{cases}$$
As, for all $\eta\in \kappa^\kappa$ and $\alpha<\kappa$, $f(\eta)(\alpha)$ depends only on $\eta\restriction\alpha$, $f$ is $1$-Lipschitz.
To see that it forms a reduction from  ${\sq X}$ to ${\sqc S}$, let $\eta,\xi$ be arbitrary elements of $\kappa^\kappa$. There are two cases to consider:

$\br$ If $\eta\sq X\xi$, then let us fix a club $C$ such that $C\cap X\subseteq \{\beta<\kappa\mid \eta(\beta)\leq\xi(\beta)\}$.
Since $\vec{\mathcal F}$ captures clubs, let us a fix a club $D\s\kappa$ such that, for all $\alpha\in D\cap S'$, $C\cap\alpha\in\mathcal F_\alpha$.
We claim that $D$ is disjoint from $\{\alpha\in S\mid f(\eta)(\alpha)>f(\xi)(\alpha)\}$,
so that indeed $f(\eta)\sqc Sf(\xi)$.
Let $\alpha\in D$ with $f(\eta)(\alpha)=1$.
Then for $W:=C\cap\alpha$, we have that $W\in\mathcal F_\alpha\s\bar{\mathcal F}_\alpha$
and $W\cap X\subseteq\{\beta<\alpha\mid \eta(\beta)\le\xi(\beta)\}$,
so that $\eta_\alpha\preceq_\alpha\eta\restriction\alpha\preceq_\alpha\xi\restriction\alpha$,
and $f(\xi)(\alpha)=1$.

$\br$ If $\neg ({\eta}\sq X{\xi})$, then $Y:=\{\beta\in X\mid \eta(\beta)>\xi(\beta)\}$ is stationary.
In effect, $T:=\{\alpha\in S''\mid \eta_\alpha=\eta\restriction\alpha\ \&\ Y\cap\alpha\in\bar{\mathcal F}_\alpha\}$ is stationary.
Now, for every $\alpha\in T$,
the set $W:=Y\cap\alpha$ is in $\bar{\mathcal F}_\alpha$ so that $f(\eta)(\alpha)=1$
and $f(\xi)(\alpha)=0$. Consequently, $\neg(f(\eta)\sqc S f(\xi))$.
\end{proof}
\begin{remark}
The notion of ``$X$ $\diamondsuit$-reflects to $S$'' from \cite[Definition~57]{FHK}
is the special case of ``$X$ $\vec{\mathcal F}$-reflects with $\diamondsuit$ to $S$''
in which, for every $\alpha\in S$, $\cf(\alpha)>\omega$ and $\mathcal F_\alpha=\cub(\alpha)$.
By \cite[Theorem~58]{FHK}, if $X$ $\diamondsuit$-reflects to $S$, then ${=^2_X}\reduc{=^2_S}$.
So the preceding theorem is an improvement not only because $\cub$ is replaced by abstract filters (allowing the case $S\s\cof(\omega)$),
but also because ${\sq X}\reduO{\sqc S}$ entails ${=_X}\reduO{=^2_S}$, so that ${=^2_X}\reduO{=_X}\reduO{=^2_S}$ and hence also ${=^2_X}\reduc{=^2_S}$.
\end{remark}
\begin{remark} Claim~\ref{claim2131} establishes another interesting fact:
If there exists a stationary $X\s\kappa$ such that $X$ $\mathfrak f$-reflects with $\diamondsuit$ to $S$, then $\diamondsuit_S$ holds.
\end{remark}

\subsection{Nilpotent reductions}
Now we turn to prove Theorem~D.
It is clear from Theorem~\ref{lemma43} that if $X$ $\mathfrak{f}$-reflects with $\diamondsuit$ to X, then there exists a 1-Lipschitz map which is different from the identity,
yet, it witnesses ${=_X}\reduO{=_X}$.
In certain cases, we can obtain a similar result with only $\mathfrak{f}$-reflection (without the need for $\diamondsuit$).

\begin{lemma}\label{onefunction}
Suppose $X,Y,Z$ are stationary subsets of $\kappa$, with $X\cap Y$ non-stationary.
\begin{enumerate}
\item If $X$ $\mathfrak{f}$-reflects to $Y$ and $Y$ $\mathfrak{f}$-reflects to $X$, then there is a function simultaneously witnessing $${=_X}\reduO{=_Y}\ \&\ {=_Y}\reduO{=_X}.$$
\item If $Z$ $\mathfrak{f}$-reflects to $Y$ and $Z$ $\mathfrak{f}$-reflects to $X$, then there is a function simultaneously witnessing $${=_Z}\reduO{=_Y}\ \&\ {=_Z}\reduO{=_X}.$$
\end{enumerate}
\end{lemma}
\begin{proof} We only prove Clause~(1). The proof of Clause~(2) is very similar.

Let $X'\s X$ and $Y'\s Y$ be stationary subsets such that there are $\vec{\mathcal H}=\langle \mathcal H_\alpha\mid\alpha\in Y'\rangle$ and  $\vec{\mathcal G}=\langle \mathcal G_\alpha\mid\alpha\in X'\rangle$
witnessing that $X$ $\mathfrak f$-reflects to $Y$, and $Y$ $\mathfrak f$-reflects to $X$, respectively. Let us define $\vec{\mathcal F}=\langle \mathcal F_\alpha\mid\alpha\in X'\symdiff Y'\rangle$ via
$$\mathcal{F}_\alpha:=\begin{cases}
\mathcal H_\alpha,&\text{if }\alpha\in X';\\
\mathcal G_\alpha,&\text{if } \alpha\in Y'.
\end{cases}$$

For every $\alpha\in X'\symdiff Y'$,
define the equivalence relation $\sim_\alpha$ over $\kappa^\alpha$ and the codes $\myceil{[\eta]_{\sim_\alpha}}$ as in Lemma~\ref{lemma53}.
Next, define a map $f:\kappa^\kappa\rightarrow\kappa^\kappa$ by letting
for all $\eta\in\kappa^\kappa$ and $\alpha<\kappa$:
$$f(\eta)(\alpha):=\begin{cases}
\myceil{[\eta\restriction\alpha]_{\sim_\alpha}},&\text{if }\alpha\in X'\symdiff Y';\\
0,&\text{otherwise}.\end{cases}$$

At this point, the analysis is the same as in the proof of Lemma~\ref{lemma53}.
\end{proof}

\begin{cor}
Suppose that $X,Y$ are stationary subsets of $\kappa$, with $X\cap Y$ non-stationary.
If $X$ $\mathfrak{f}$-reflects to $Y$ and $Y$ $\mathfrak{f}$-reflects to $X$, then there is a function $f:\kappa^\kappa\rightarrow\kappa^\kappa$ different from the identity witnessing ${=_X}\reduO{=_X}$.
\end{cor}
\begin{proof} Take the square of a function witnessing Lemma~\ref{onefunction}(1).
\end{proof}

\begin{cor} Suppose $\textsf{V}=\textsf{L}$ and $\mu\in\reg(\kappa)$.
Then there is a function simultaneously witnessing ${=_{\kappa\cap \cof (\mu)}}\reduO{=_{\kappa\cap \cof (\nu)}}$ for all $\nu\in\reg(\kappa)$.
\end{cor}
\begin{proof}
By Lemma~\ref{onefunction} and Corollary~\ref{V=L-fref} below.
\end{proof}

\section{Strong and simultaneous forms of filter reflection}\label{Section3}

\begin{defn} For stationary sets $X,S\s\kappa$ and a cardinal $\theta\le\kappa$,
the principle $\mathfrak{f}$-$\refl(\theta,X,S)$ asserts
the existence of a sequence $\vec{\mathcal F}=\langle\mathcal F_\alpha\mid\alpha\in S'\rangle$
with $S'\s S$,
such that each $\mathcal F_\alpha$ is a filter over $\alpha$,
$\vec{\mathcal F}$ captures clubs, and,
for every sequence $\langle Y_i\mid i<\theta\rangle$ of stationary subsets of $X$,
the set $\{\alpha\in S'\mid \forall i<\max\{\theta,\alpha\}~(Y_i\cap\alpha\in\mathcal F_\alpha^+)\}$ is stationary.
\end{defn}
\begin{remark}
\begin{enumerate}
\item In the special case $\theta=1$, we shall omit $\theta$, writing $\mathfrak{f}$-$\refl(X,S)$.
Note that the latter coincides with the notion of ``$X$ $\mathfrak f$-reflects to $S$".
\item In the special case in which $\mathcal F_\alpha=\cub(\alpha)$ for all $\alpha\in S$,
we shall omit $\mathfrak f$, writing, e.g., ``$\refl(\theta,X,S)$" and
``$X$ reflects with $\diamondsuit$ to $S$".
\end{enumerate}
\end{remark}

Another standard notion of reflection is \emph{Friedman's problem}, $\fp(\kappa)$, asserting that
every stationary subset of $\kappa\cap\cof(\omega)$ contains a closed copy of $\omega_1$.

\begin{lemma} If $\fp(\kappa)$ holds, then
$\kappa\cap\cof(\omega)$ strongly $\mathfrak f$-reflects to $\kappa\cap\cof(\omega_1)$.
\end{lemma}
\begin{proof} For each $\alpha\in\kappa\cap\cof(\omega_1)$, let $\mathcal F_\alpha:=\cub(\alpha)$.
\end{proof}

The following forms a partial converse to the implication $(1)\implies(2)$ of Proposition~\ref{prop23}.

\begin{prop} Suppose $X$ strongly $\mathfrak f$-reflects to $S$.

If $\diamondsuit_X$ holds, then so does $\diamondsuit_S$.
\end{prop}
\begin{proof} Fix $\vec{\mathcal F}=\langle \mathcal F_\alpha\mid\alpha \in S'\rangle$ with $S'\subseteq S$
such that $X$ strongly $\vec{\mathcal F}$-reflects to $S'$.
\begin{claim}\label{claim341} The set $A:=\{ \alpha\in S'\mid \exists B\in\mathcal F_\alpha(\sup(B)<\alpha)\}$ is non-stationary.
\end{claim}
\begin{proof} Suppose not. Fix $\epsilon<\kappa$ for which
$A_\epsilon:=\{ \alpha\in S'\mid \exists B\in\mathcal F_\alpha(\sup(B)=\epsilon)\}$
is stationary. Now, consider the club $C:=\kappa\setminus(\epsilon+1)$.
As $\vec{\mathcal F}$ captures clubs and $A_\epsilon$ is stationary, there must exist $\alpha\in A_\epsilon$ such that $C\cap\alpha\in\mathcal F_\alpha$.
As $\alpha\in A_\epsilon$, let us pick $B\in\mathcal F_\alpha$ with $\sup(B)=\epsilon$.
Then $(C\cap\alpha)\cap B\in\mathcal F_\alpha$, contradicting the fact the former is empty.
\end{proof}

Suppose that $\diamondsuit_X$ holds.
Fix a $\diamondsuit_X$-sequence $\langle Z_\beta\mid \beta\in X\rangle$.
For every $\alpha\in S'\setminus A$, let $\mathcal Z_\alpha:=\{ Z\s\alpha\mid \{ \beta\in X\cap\alpha\mid Z\cap\beta=Z_\beta\}\in\mathcal F_\alpha\}$.
\begin{claim}$\mathcal Z_\alpha$ contains at most a single set.
\end{claim}
\begin{proof} Towards a contradiction, suppose that $Z\neq Z'$ are elements of $\mathcal Z_\alpha$.
Let $B:=\{ \beta\in X\cap\alpha\mid Z\cap\beta=Z_\beta\}$ and $B':=\{ \beta\in X\cap\alpha\mid Z'\cap\beta=Z_\beta\}$.
Fix $\zeta\in Z\symdiff Z'$. As $Z,Z'\in\mathcal Z_\alpha$, $B\cap B'\in\mathcal F_\alpha$.
As $\alpha\notin A$, we may find some $\beta\in B\cap B'$ above $\zeta$.
But then $Z\cap\beta=Z'\cap\beta$, contradicting the fact that $\zeta\in (Z\cap\beta)\symdiff(Z'\cap\beta)$.
\end{proof}

For each $\alpha\in S$, fix a subset $Z_\alpha\s \alpha$ such that,
if $\alpha\in S'\setminus A$, then $\mathcal Z_\alpha\s\{Z_\alpha\}$.
To see that $\langle Z_\alpha\mid\alpha\in S\rangle$ is a $\diamondsuit_{S}$-sequence,
let $Z$ be an arbitrary subset of $\kappa$.
Consider the stationary set $Y:=\{\beta\in X\mid Z\cap\beta=Z_\beta\}$.
Pick $\alpha\in S'\cap\acc^+(Y)\setminus A$ such that $Y\cap\alpha\in\mathcal F_\alpha$. Then $Z_\alpha=Z\cap\alpha$.
\end{proof}

The preceding provides a way of separating $\mathfrak f$-reflection from strong $\mathfrak f$-reflection.
Indeed, by the main result of \cite{MR1164732}, it is consistent that for some weakly compact cardinal $\kappa$, $\diamondsuit_\kappa$ holds, but $\diamondsuit_{\reg(\kappa)}$ fails. It follows that in any such model, $\kappa$ does not strongly $\mathfrak f$-reflects to $\reg(\kappa)$,
while $\refl(\kappa,\reg(\kappa))$  does hold (see Lemma~\ref{wc}(1) below).

A similar configuration can also be obtained at the level of accessible cardinals as small as $\aleph_2$.
In \cite[\S5]{jingzhang}, assuming the consistency of a weakly compact cardinal, Zhang constructed a model of $\gch$ in which $\refl(S^2_0,S^2_1)$ holds, but $\diamondsuit_{S^2_1}$ fails.
By \cite[Lemma~2.1]{MR485361}, $\gch$ implies that $\diamondsuit_{X}$ holds for any stationary $X\s S^2_0$. Thus:
\begin{cor} Assuming the consistency of a weakly compact cardinal, it is consistent that the two hold together:
\begin{itemize}
\item Every stationary subset of $S^2_0$ reflects in $S^2_1$;
\item There exists no stationary subset of $S^2_0$ that strongly $\mathfrak f$-reflects to $S^2_1$.\qed
\end{itemize}
\end{cor}

The rest of this section is motivated by Theorem~\ref{lemma43}.

Evidently, if $X$ $\mathfrak{f}$-reflects with $\diamondsuit$ to $S$, then $\mathfrak{f}$-$\refl(X,S)$ and $\diamondsuit_S$ both hold.
The next lemma deals with the converse implication.

\begin{lemma}\label{addonecohen}
Let $X\s\kappa$ and $S\s\kappa\cap\cof({>}\omega)$.
For $\mathbb P:=\add(\kappa,1)$, the following are equivalent:
\begin{enumerate}
\item $V^{\mathbb P}\models\refl(X,S)$;
\item $V^{\mathbb P}\models X$ reflects with $\diamondsuit$ to $S$.
\end{enumerate}
\end{lemma}
\begin{proof} It is clear from the definition that $(2)\implies(1)$.

$(1)\implies(2)$: 
For all $\alpha<\kappa$, we define a $\mathbb P$-name for a subset of $\alpha$, as follows: 
$$\dot D_\alpha:=\{ (\check\beta,p)\mid \beta<\alpha,~ p\in{2^{<\kappa}}\setminus{2^{\le\alpha+\beta}},~p(\alpha+\beta)=1\}.$$
Now, suppose that $\dot Y$ is a $\mathbb P$-name, $p\in\mathbb P$ and $p$ forces that $\dot Y$ is a stationary subset of $X$;
we shall find an extension $r$ of $p$ and an ordinal $\alpha\in S$ such that $r$ forces that $\dot D_\alpha$ is a stationary subset of $\alpha$ which is also an initial segment of $\dot Y$.

\begin{claim} There are a condition $q$ extending $p$ and an ordinal $\alpha\in S$ such that:
\begin{itemize}
\item $\dom(q)=\alpha$;
\item $q$ decides $\dot Y$ up to $\alpha$, and decides it to be a stationary subset of $\alpha$.
\end{itemize}
\end{claim}
\begin{proof} Let $G$ be a $\mathbb P$-generic over $V$, with $p\in G$.
Work in $V[G]$. Let $g:=\bigcup G$. As $\mathbb P$ does not add bounded subsets of $\kappa$,
we may define a function $f:\kappa\rightarrow\kappa$ such that, for all $\epsilon<\kappa$, $g\restriction(f(\epsilon))$ decides $\dot Y$ up to $\epsilon$.
Consider the club $C:=\{\alpha<\kappa\mid f[\alpha]\s\alpha\}$, and note that, for all $\alpha\in C$,
$g\restriction\alpha$ decides $\dot Y$ up to $\alpha$.
As $V^{\mathbb P}\models\refl(X,S)$,
$p\in G$, and $p$ forces that $\dot Y$ is a stationary subset of $X$,
$R:=\{\alpha\in C\cap S\mid (\dot Y_G)\cap\alpha\text{ is stationary in }\alpha\}$ is stationary in $V[G]$.
Fix $\alpha\in R$ with $\alpha>\dom(p)$. Then $\alpha$ and $q:=g\restriction\alpha$ are as sought.
\end{proof}

Let $\alpha$ and $q$ be given by the claim. Fix a stationary $d\s\alpha$ such that $q$ forces that $\dot Y$ up to $\alpha$ is equal to $\check d$.
Define a function $r:\alpha+\alpha\rightarrow2$ by letting, for all $\epsilon<\alpha+\alpha$,
$$r(\epsilon):=\begin{cases}
q(\epsilon),&\text{if }\epsilon<\alpha;\\
1,&\text{if }\epsilon=\alpha+\beta\ \&\ \beta\in d;\\
0,&\text{otherwise}.
\end{cases}$$

Then $r$ extends $p$ and forces that $\dot D_\alpha$ is a stationary subset of $\alpha$ which is also an initial segment of $\dot Y$.
\end{proof}

The following is well-known, the second item is pointed out in \cite{magidor1982reflecting}.
As we will need the proof later on (in proving Theorem~\ref{dl-to-fref}), we do include it.
\begin{lemma}\label{wc} Suppose that $\kappa$ is weakly compact. Then:
\begin{enumerate}
\item $\refl(\kappa,\kappa,\reg(\kappa))$ holds;
\item For every $\lambda\in\reg(\kappa)$, in the forcing extension by $\Col(\lambda,{<}\kappa)$,
$\kappa=\lambda^+$ and
$\refl(\kappa,\kappa\cap\cof({<}\lambda),\kappa)$ holds.
\end{enumerate}
\end{lemma}
\begin{proof} (1)
We shall consider a few first-order sentences in the language with unary predicate symbols $\mathbb O$ and $\mathbb D$, and
binary predicate symbols $\epsilon, \mathbb A$ and $\mathbb F$.
Specifically, let $\varphi_0$ and $\varphi_1$ denote first-order sentences such that:
\begin{itemize}
\item $\langle V_\alpha,{\in},O,F\rangle\models\varphi_0$ iff $F$ is a function from some $\beta\in O$ to $O$;
\item $\langle V_\alpha,{\in},O,F\rangle\models\varphi_1$ iff $F$ is a function with $\im(F)\s\gamma$ for some $\gamma\in O$.
\end{itemize}
Evidently, $\Phi:=\forall F(\varphi_0\rightarrow\varphi_1)$ is a $\Pi^1_1$-sentence such that $\langle V_\alpha,{\in},{\ord}\cap\alpha\rangle\models\phi$
iff $\alpha$ is a regular cardinal.

Also, let $\psi$ be a first-order sentence such that, for every $\alpha\in\cof({>}\omega)$, $\langle V_\alpha,{\in},{\ord}\cap\alpha,D\rangle\models\psi$ iff $D$ is a club in $\alpha$.
Then, let $\Psi$ be the following $\Pi^{1}_{1}$-sentence:
$$\forall D\forall \iota((\psi(D)\wedge\mathbb O(\iota))\rightarrow \exists \beta( D(\beta)\wedge  (\mathbb A(\iota,\beta)) )).$$

Now, to verify $\refl(\kappa,\kappa,\reg(\kappa))$,
fix an arbitrary sequence  $\langle Y_\iota \mid \iota < \kappa \rangle $ of stationary subsets of $\kappa$,
and we shall find an $\alpha\in\reg(\kappa)$ such that $Y_\iota\cap\alpha$ is stationary for all $\iota<\alpha$.

Set $A:=\{(\iota,\beta) \mid \iota <\kappa\ \&\ \beta \in Y_\iota \}$.
A moment reflection makes it clear that $$\langle V_\kappa,{\in},{\ord}\cap\kappa,A\rangle\models\Phi\wedge\Psi.$$
As weak compactness is equivalent to $\Pi_{1}^{1}$-indescribability (cf.~\cite[Theorem~6.4]{Kana}),
there exists an uncountable $\alpha<\kappa$ such that:
$$ \langle V_{\alpha}, \in,{\ord}\cap\alpha,A \cap (\alpha \times \alpha) \rangle \models \Phi\wedge\Psi.$$
Clearly, $\alpha$ is as sought.

(2) Let $\lambda\in\reg(\kappa)$.
As in \cite[Claim 2.11.1]{AHKM},
we work with a partial order $\mathbb P$ which is isomorphic to $\Col(\lambda, {<}\kappa)$,
but, in addition,  $\mathbb P\s V_\kappa$. Namely, $\mathbb P=(P,\le)$, where
$$ P:=\{ r\restriction(\sup(\supp(r))+1)  \mid r \in \Col(\lambda, {<}\kappa) \},$$
and $q\le p$ iff $q\supseteq p$.
Note that, for every $\alpha\in\reg(\kappa)$ above $\lambda$,
$\mathbb P$ is also isomorphic to $\mathbb P_\alpha\times\mathbb P^{\ge\alpha}$, where
\begin{itemize}
\item $\mathbb{P}_\alpha:=(P_\alpha,\le)$,  with
$P_\alpha:=\{ r\restriction(\sup(\supp(r))+1)  \mid r \in \Col(\lambda, {<}\alpha)\}$, and
\item $\mathbb{P}^{\ge\alpha}:=(P^{\ge\alpha},\le)$,  with
$P^{\ge\alpha}:=\{ r\restriction(\sup(\supp(r))+1)  \mid r \in \Col(\lambda, [\alpha,\kappa)) \}$.
\end{itemize}

Now, suppose $p_0\in P$ and $\dot{f}$ is a $\mathbb P$-name such that $p_0$ forces
$\dot{f}$ is a function with domain $\kappa$, and, for each $\iota<\kappa$, $\dot f(\iota)$ is a stationary subset of $\kappa\cap\cof({<}\lambda)$.

Define a set $H$ to consist of all quadruples $(\iota,\beta,p,q)$
such that:
\begin{itemize}
\item $\iota<\kappa$,
\item $\beta\in\kappa\cap\cof({<}\lambda)$,
\item $p,q\in V_\kappa$;
\item if $p\in P$ and $p\le p_0$, then $q\in P$, $q\le p$ and  $q\forces_{\mathbb P}\check\beta\in\dot f(\check\iota)$.
\end{itemize}
Let $\op(\dot x,\dot y)$ denote the canonical name for the ordered pair whose left element is $\dot x$ and right element is $\dot y$,
and set $\dot A:=\{(\op(\check\iota,\check\beta),q)\mid \exists p\le p_0~(\alpha,\beta,p,q)\in H\}$,
so that  the interpretation of $\dot A$ plays the role of the set $A$ from the proof of the previous clause. Note that $p_0\forces \dot{A}=\dot{f}$.

Let $\psi$ be a first-order sentence as in the previous clause.
Then, let $\Psi$ be the following $\Pi^{1}_{1}$-sentence:
$$\forall p\forall D\forall\iota((\psi(D)\wedge\mathbb O(\iota))\rightarrow \exists q\exists \beta( D(\beta)\wedge  (\mathbb H(\iota,\beta,p,q))).$$

Recalling that for any condition $p\le p_0$, any club $D$ in $\kappa$, and any $\iota<\kappa$, there is a condition $q\le p$ deciding the existence of an ordinal in $\dot f(\iota)$ and in $D$, we have:
$$\langle V_\kappa,{\in},{\ord}\cap\kappa,H\rangle\models \Psi.$$

Since $\kappa$ is $\Pi^{1}_{1}$-indescribable, we may fix a strongly inaccessible cardinal $\alpha<\kappa$ such that
$$(V_{\alpha},\in, {\ord}\cap\alpha, H \cap V_{\alpha}) \models \Psi.$$

As $\mathbb P_\alpha$ has the $\alpha$-cc, every club in $\alpha$ in $V^{\mathbb P_\alpha}$ covers a club in $\alpha$ from $V$. Therefore:
$$p_{0} \forces_{\mathbb P_\alpha} \langle \dot{f}(\iota)\cap \alpha \mid \iota < \alpha \rangle  \text{ is an } \alpha\text{-sequence of stationary subsets of } \alpha\cap\cof({<}\lambda).$$

By \cite[Theorem~20]{Sh:108} and \cite[Lemma~4.4(1)]{Sh:351} (see also \cite[\S2]{Todd}), for $\lambda$ regular, every stationary subset of $\lambda^+\cap\cof({<}\lambda)$ is preserved by a ${<}\lambda$-closed notion of forcing.
In $V^{\mathbb P_\alpha}$, $\alpha=\lambda^+$ for the regular cardinal $\lambda$,
and $\mathbb P^{\ge\alpha}$ is ${<}\lambda$-closed. Therefore:
\[p_{0} \forces_{\mathbb P_\alpha\times\mathbb P^{\ge\alpha}} \langle \dot{f}(\iota)\cap \alpha \mid \iota < \alpha \rangle  \text{ is an } \alpha\text{-sequence of stationary subsets of } \alpha .\qedhere\]
\end{proof}

The upcoming corollary was announced first by Shelah and V\"a\"an\"anen in \cite{SV05} without a proof. It amounts to the special case of ``$\kappa$ reflects with $\diamondsuit$ to $\reg(\kappa)$'', when $\kappa$ is weakly compact (see \cite[Definition~9]{SV05}).

\begin{cor}[{Hellsten, \cite[Lemma~5.2.3]{hellsten2003diamonds}}]\label{hellsten} If $\kappa$ is weakly compact, then, in some cofinality-preserving forcing extension,
$\kappa$ reflects with $\diamondsuit$ to $\reg(\kappa)$.
\end{cor}
\begin{proof}
Recall that Silver proved that for every weakly compact cardinal $\kappa$,
there is a forcing extension $V[G]$ such that $V[G][H]\models\kappa\text{ is weakly compact}$,
whenever $H$ is an $\add(\kappa,1)$-generic over $V[G]$.
(cf.~\cite[Example~16.2]{MR2768691}).
Now, appeal to Lemma~\ref{wc}(1) and Lemma~\ref{addonecohen}.
\end{proof}

The preceding took care of weakly compact cardinals.
Building on the work of Hayut and Lambie-Hanson from \cite{MR3730566},
it is also possible to obtain reflection at the level of successor of singulars and small inaccessibles.
It is of course also possible to obtain reflection with diamond at the level of successors of regulars,
but this is easier, and will be done in Corollary~\ref{succregulars} below.

\begin{cor}
\begin{enumerate}
\item  If the existence of infinitely many supercompact cardinals is consistent, then it is consistent that $\aleph_{\omega+1}$ reflects with $\diamondsuit$ to $\aleph_{\omega+1}$;
\item  If the existence of an inaccessible limit of supercompact cardinals is consistent, then it is consistent that, letting $\kappa$ be the least inaccessible cardinal, $\kappa$ reflects with $\diamondsuit$ to  $\kappa$.
\end{enumerate}
\end{cor}
\begin{proof} Following \cite{MR3730566}, for a stationary $X\s\kappa$,
we let $\refl^*(X)$ assert that whenever $\mathbb{P}$ is a ${<}\kappa$-closed-directed forcing notion of size $\le\kappa$,
$\Vdash_\mathbb{P} \refl(X,\kappa)$.
In particular, if $\refl^*(X)$ holds and $\kappa^{<\kappa}=\kappa$,
then for $\mathbb P:=\add(\kappa,1)$,
we would have that $V^{\mathbb P}\models\refl(X,\kappa)$,
and then by Lemma~\ref{addonecohen}, furthermore, $V^{\mathbb P}\models X$ reflects with $\diamondsuit$ to $\kappa$.

Now, by  \cite[Theorem~3.23]{MR3730566}, the hypothesis of Clause~(1) yields the consistency of $\refl^*(\aleph_{\omega+1})$ holds.
Likewise, by  \cite[Theorem~3.24]{MR3730566}, the hypothesis of Clause~(2) yields the consistency of the statement that,
there is an inaccessible cardinal and the least inaccessible cardinal $\kappa$ satisfies $\refl^*(\kappa)$.
\end{proof}

Another approach for adjoining diamond to reflection is as follows.

\begin{lemma}\label{diag}
Suppose $\mathfrak{f}$-$\refl(\kappa,X,S)$ holds for stationary subsets $X,S$ of $\kappa$.
If $\diamondsuit^*_S$ holds, then $X$ $\mathfrak{f}$-reflects with $\diamondsuit$ to $S$.
\end{lemma}
\begin{proof} Let $\vec{\mathcal{F}}=\langle \mathcal F_\alpha\mid\alpha\in S'\rangle$ be a sequence witnessing $\mathfrak{f}$-$\refl(\kappa,X,S')$, so that $S'\s S$.
Suppose that $\diamondsuit^*_{S}$ (or just $\diamondsuit^*_{S'}$) holds.
It follows that we may fix a matrix $\langle Z_{\alpha}^i\mid i<\alpha<\kappa\rangle$
such that, for every $Z\s\kappa$, for club many $\alpha\in S'$, there is $i<\alpha$ with $Z_{\alpha}^i=Z\cap\alpha$.
Fix a bijection $\pi:\kappa\times\kappa\leftrightarrow\kappa$.
For any pair $i<\alpha<\kappa$, let $Y_\alpha^i:=\{\beta<\alpha\mid \pi(\beta,i)\in Z_{\alpha}^i\}$.
\begin{claim} There exists $i<\kappa$ such that for every stationary $Y\s X$, the set
$\{\alpha\in S'\mid Y_\alpha^i=Y\cap\alpha ~ \& ~ Y\cap \alpha\in \mathcal{F}_\alpha^+\}$ is stationary.
\end{claim}
\begin{proof}Suppose not. Then, for every $i<\kappa$, we may find a stationary $Y_i\s X$
for which $\{\alpha\in S\mid Y_\alpha^i=Y_i\cap\alpha ~ \& ~ Y_i\cap \alpha\in \mathcal{F}_\alpha^+\}$ is non-stationary.
Let $Z:=\{ \pi(\beta,i)\mid i<\kappa, \beta\in Y_i\}$.
Fix a club $C\s\kappa$ such that, for all $\alpha\in C\cap S$:
\begin{itemize}
\item $\pi[\alpha\times\alpha]=\alpha$, and
\item there exits $i<\alpha$ with $Z_{\alpha}^i=Z\cap\alpha$.
\end{itemize}
By the hypothesis, $T:=\{\alpha\in S'\cap C\mid \forall i<\alpha~(Y_i\cap\alpha\in \mathcal{F}_\alpha^+)\}$ must be stationary.
By Fodor's lemma, we may fix $i^*<\kappa$ and a stationary $T'\s T$ such that, for all $\alpha\in T'$, $Z_{\alpha}^{i^*}=Z\cap\alpha$.
For all $\alpha\in T'$, we have:
\begin{itemize}
\item $\alpha\in T$, so that, in particular, $Y_{i^*}\cap\alpha\in\mathcal F_\alpha^+$;
\item $\alpha\in C$, so that $Y_\alpha^{i^*}=\{\beta<\alpha\mid \pi(\beta,i^*)\in Z_{\alpha}^{i^*}\}=\{\beta<\alpha\mid \pi(\beta,i^*)\in Z\cap\alpha\}=Y_{i^*}\cap\alpha.$
\end{itemize}

It thus follows that $\{\alpha\in S'\mid Y_\alpha^{i^*}=Y_{i^*}\cap\alpha ~ \& ~Y_{i^*}\cap\alpha\in \mathcal{F}_\alpha^+\}$ covers the stationary set $T'$,
contradicting the choice of $Y_{i^*}$.
\end{proof}

This clearly completes the proof.
\end{proof}

\begin{remark}\label{diag2}
\begin{enumerate}
\item Note that by the definition of $\diamondsuit^*_S$, if $\kappa=\lambda^+$ is a successor cardinal, then the above argument establishes that for stationary subsets $X,S$ of $\kappa$:
If $\mathfrak{f}$-$\refl(\lambda,X,S)$ and $\diamondsuit^*_S$ both hold, then $X$ $\mathfrak{f}$-reflects with $\diamondsuit$ to $S$.
\item  Note that the same proof establishes the corresponding fact for the $\mathfrak f$-free (that is, genuine) versions of reflection.
\end{enumerate}
\end{remark}

\begin{cor} Suppose that $\lambda$ is a regular uncountable cardinal, and $\diamondsuit^*_{\lambda^+\cap\cof(\lambda)}$ holds.
For every stationary $X\s \lambda^+\cap\cof({<}\lambda)$, if $\refl(\lambda,X,\lambda^+)$ holds, then $X$ reflects with $\diamondsuit$ to $\lambda^+\cap\cof(\lambda)$.
\end{cor}
\begin{proof} Suppose we are given $X\s \lambda^+\cap\cof({<}\lambda)$ for which $\mathfrak{f}$-$\refl(\lambda,X,\lambda^+)$ holds.
By Remark~\ref{diag2}, it suffices to prove that  $\refl(\lambda,X,\lambda^+\cap\cof(\lambda))$ holds.
For this, let $\langle Y_i\mid i<\lambda\rangle$ be some sequence of stationary subsets of $X$.
Fix an arbitrary partition $\vec X=\langle X_i\mid i<\lambda\rangle$ of $X$ into stationary sets.
As $\refl(\lambda,X,\lambda^+)$ holds, the set $$A:=\{\alpha\in\lambda^+\mid \forall i<\lambda(Y_i\cap\alpha\text{ and }X_i\cap\alpha\text{ are stationary in }  \alpha)\}$$ is stationary.
As the elements of $\vec X$ are pairwise disjoint, it follows that $A\s\lambda^+\cap\cof(\lambda)$.
\end{proof}

\begin{cor}\label{succregulars} Suppose that $\kappa$ is a weakly compact cardinal.

For every $\lambda\in\reg(\kappa)$, in the forcing extension by $\Col(\lambda,{<}\kappa)$,
$\lambda^+\cap\cof({<}\lambda)$ reflects with $\diamondsuit$ to $\lambda^+\cap\cof(\lambda)$.
\end{cor}
\begin{proof}
Let $\lambda\in\reg(\kappa)$, and work in the forcing extension by $\Col(\lambda,{<}\kappa)$.
As the proof of \cite[Example~1.26]{brodsky2017microscopic} shows, $\diamondsuit^*_{\lambda^+}$ holds.
In addition, by Lemma~\ref{wc}(2), $\refl(\lambda^+,\lambda^+\cap\cof({<}\lambda),\lambda^+)$ holds.
So, by the preceding corollary, $\lambda^+\cap\cof({<}\lambda)$ reflects with $\diamondsuit$ to $\lambda^+\cap\cof(\lambda)$.
\end{proof}
\begin{remark} It thus follows from Theorem~\ref{lemma43} that in the model of the preceding,
for every $\mu\in\reg(\lambda)$,
${\sq{\lambda^+\cap\cof(\mu)}}\reduO{\sqc{\lambda^+\cap\cof(\lambda)}}$, and hence also ${=_{\lambda^+\cap\cof(\mu)}}\reduO{=^2_{\lambda^+\cap\cof(\lambda)}}$.
This improves \cite[Theorem~55]{FHK},
as the result is not limited to double successors, and as we do not need to assume that our ground model is $\mathsf{L}$.
\end{remark}

A variation of the proof of Lemma~\ref{diag} yields the following:

\begin{lemma} \label{MM} Suppose Martin's Maximum $(\mm)$ holds,
$\kappa\ge\aleph_2$, $X\s\kappa\cap\cof(\omega)$ is stationary, and $S=\kappa\cap\cof(\omega_1)$.
If $\diamondsuit_{X}$ holds,\footnote{Recall that by a theorem of Shelah \cite{Sh:922},
if $\kappa=\kappa^{<\kappa}$ is the successor of a cardinal uncountable cofinality,
then $\diamondsuit_X$ holds for every stationary $X\s\kappa\cap\cof(\omega)$.}
then $X$ reflects with $\diamondsuit$ to $S$.
\end{lemma}
\begin{proof}
Suppose that $\diamondsuit_X$ holds, as witnessed by $\langle Z_\gamma\mid\gamma\in X\rangle$.
For every $Z\s\kappa$, let $$G_Z:=\{\gamma\in\sup(Z)\cap X\mid Z\cap\gamma=Z_\gamma\}.$$
Let $\alpha\in S$ be arbitrary. Fix a strictly increasing function $\pi_\alpha:\omega_1\rightarrow\alpha$ whose image is a club in $\alpha$,
and then let $$\mathcal Z_\alpha:=\{ Z\s\alpha\mid \pi_\alpha^{-1}[G_Z]\text{ is stationary in }\omega_1\}.$$
Evidently, for all two distinct $Z,Z'\in \mathcal Z_\alpha$, $\pi_\alpha^{-1}[G_Z]$ and $\pi_\alpha^{-1}[G_{Z'}]$ are almost disjoint stationary subsets of $\omega_1$.
As $\mm$ implies that $\ns_{\omega_1}$ is saturated, we infer that $|\mathcal Z_\alpha|\le\aleph_1$.
Thus, let $\{ Z_\alpha^i\mid i<\omega_1\}$ be some enumeration (possibly, with repetitions) of $\mathcal Z_\alpha$.
Fix a bijection $\pi:\kappa\times\omega_1\leftrightarrow\kappa$.
For each $\alpha\in S$ and $i<\omega_1$, let $Y_\alpha^i:=\{\beta<\alpha\mid \pi(\beta,i)\in Z_{\alpha}^i\}$.
\begin{claim} There exists $i<\omega_1$ such that, for every stationary $Y\s\kappa$,
$\{ \alpha\in S \cap \tr(Y) \mid Y\cap\alpha=Y^i_\alpha\}$ is stationary.
\end{claim}
\begin{proof}Suppose not.  For every $i<\omega_1$, fix a stationary $Y_i\s X$
for which $\{\alpha\in S \cap \tr(Y_{i}) \mid Y_i\cap\alpha=Y_\alpha^i\}$ is non-stationary.
Let $Z:=\{ \pi(\beta,i)\mid i<\omega_1, \beta\in Y_i\}$.
Consider the stationary set $G:=\{\gamma\in X\mid Z\cap\gamma=Z_\gamma\}$
and the club $C:=\{\alpha\in\acc^+(Z)\mid \pi[\alpha\times\omega_1]=\alpha\}$.
As $\mm$ implies $\refl(\omega_1,X,S)$, the following set is stationary
$$S':=\{\alpha\in C\cap S\mid G\cap\alpha\text{ is stationary, and, for all }i<\omega_1~Y_i\cap\alpha\text{ is statioanry}\}.$$

For every $\alpha\in S'$, since  $\alpha\in\tr(G)\cap\acc^+(Z)$, $G_{Z\cap\alpha}$ covers the stationary set $G\cap\alpha$,
so there exists $i<\omega_1$ such that $Z_\alpha^i=Z\cap\alpha$.
Now, fix $i^*<\omega_1$ and a stationary $T\s S'$ such that, for all $\alpha\in T$, $Z_{\alpha}^{i^*}=Z\cap\alpha$.
For all $\alpha\in T$, we have:
\begin{itemize}
\item $\alpha\in S'$, and hence $\alpha\in S\cap\tr(Y_{i^*})$;
\item $\alpha\in C$, and hence $Y^{i^*}_\alpha=\{\beta<\alpha\mid \pi(\beta,i^*)\in Z_{\alpha}^{i^*}\}=\{\beta<\alpha\mid \pi(\beta,i^*)\in Z\cap\alpha\}=Y_{i^*}\cap\alpha.$
\end{itemize}

So $\{\alpha\in S \cap \tr(Y_{i^{*}}) \mid Y_{i^*}\cap\alpha=Y_\alpha^{i^*}\}$ covers the stationary set $T$,
contradicting the choice of $Y_{i^*}$.
\end{proof}
Let $i$ be given by the preceding claim.
Then $\langle Y_\alpha^i\mid \alpha\in S\rangle$ witnesses that $X$ $\langle \cub(\alpha)\mid\alpha\in S\rangle$-reflects with $\diamondsuit$ to $S$.
\end{proof}

\begin{defn}\label{pointer-ineffable} A stationary subset $S$ of $\kappa$ is said to be:
\begin{enumerate}
\item \emph{ineffable} iff for every sequence $\langle A_\alpha\mid\alpha\in S\rangle$, there exists $A\s\kappa$ for which $\{\alpha\in S\mid A\cap\alpha=A_\alpha\cap\alpha\}$ is stationary.
\item \emph{weakly ineffable} iff for every sequence $\langle A_\alpha\mid\alpha\in S\rangle$, there exists $A\s\kappa$ for which $\{\alpha\in S\mid A\cap\alpha=A_\alpha\cap\alpha\}$ is cofinal in $\kappa$.
\item \emph{weakly compact} iff for every $\Pi^1_1$-sentence $\phi$ and every $A\s V_\kappa$ such that $\langle V_\kappa,\in,A\rangle\models\phi$, there exists $\alpha\in S$ such that $\langle V_\alpha,\in,A\cap V_\alpha\rangle\models\phi$.
\end{enumerate}
\end{defn}

\begin{defn}[Sun, \cite{sun}] For a weakly compact subset $S\s\kappa$, $\diamondsuit^1_S$ asserts the existence of a sequence $\langle Z_\alpha\mid\alpha\in S\rangle$ such that,
for every $Z\s\kappa$, the set $\{\alpha\in S\mid Z\cap\alpha=Z_\alpha\}$ is weakly compact.
\end{defn}

The proof of \cite[Theorem~2.11]{sun} makes clear that, for every weakly ineffable $S\s\kappa$, $\diamondsuit^1_S$ holds.
It is also easy to see that $\diamondsuit^1_S$ implies that $\kappa$ reflects with $\diamondsuit$ to $S$. Therefore:

\begin{cor}\label{cor45} For every weakly ineffable $S\s\kappa$, $\kappa$ reflects with $\diamondsuit$ to $S$.\qed
\end{cor}

We conclude this section by proving that $\diamondsuit$-reflection is equivalent to various seemingly stronger statements.
For instance, the concept of Clause~(2) of the next lemma
is implicit in \cite{AHKM}, as the principle $\text{WC}^*_\kappa$ from \cite[Lemma~3.4]{AHKM} is equivalent to the instance $X:=\kappa$ and $S:=\reg(\kappa)$.
Likewise, the question of whether Clause~(1) implies Clause~(3) is implicit in the statement of \cite[Claim~2.11.1]{AHKM}.
The fact that the two clauses are equivalent allows to reduce the hypothesis of ``there is a $\Pi_1^{\lambda^+}$-indescribable cardinal''
of \cite[Theorem~2.11]{AHKM} down to ``there is a weakly compact cardinal'' via Corollary~\ref{succregulars} above.

\begin{lemma}\label{diamond-equiv} Let $X\s\kappa$ and $S\s\kappa\cap\cof({>}\omega)$.
Then the following are equivalent:
\begin{enumerate}
\item  $X$ reflects with $\diamondsuit$ to $S$;
\item there exists a sequence $\langle f_\alpha\mid\alpha\in S\rangle$
such that, for every $g\in\kappa^\kappa$ and every stationary $Y\s X$, the set $\{\alpha\in S\cap\tr(Y)\mid g\restriction\alpha=f_\alpha\}$ is stationary;
\item there exists a partition $\langle S_i\mid i<\kappa\rangle$ of $S$ such that, for all $i<\kappa$, $X$ reflects to $S_i$ with $\diamondsuit$.
\end{enumerate}
\end{lemma}
\begin{proof} $(1)\implies(2)$: This is a special case of the proof of Claim~\ref{claim2131}.

$(2)\implies(3)$ Let $\langle f_\alpha\mid\alpha\in S\rangle$ be as in Clause~(3). Without loss of generality, for all $\alpha\in S$, $f_\alpha$ is a function from $\alpha$ to $\kappa$.
For all $i<\kappa$, let $S_i:=\{\alpha\in S\mid f_\alpha(0)=i\}$, so that $\langle S_i\mid i<\kappa\rangle$ is a partition of $S$.
For each $\alpha\in S$, let $$Y_\alpha:=\{\beta<\kappa\mid (\beta=0\wedge f_\alpha(1)\text{ is odd})\vee(\beta=1\wedge f_\alpha(1)\ge 2)\vee(\beta>1\wedge f_\alpha(\beta)=4)\}.$$

Let $i<\kappa$ be arbitrary.
We claim that $\langle Y_\alpha\mid\alpha\in S_i\rangle$ witnesses that $X$ reflects with $\diamondsuit$ to $S_i$.
To see this, fix an arbitrary stationary subset $Y$ of $X$. Define a function $g:\kappa\rightarrow\kappa$ as follows:
$$g(\beta):=\begin{cases}
i,&\text{if }\beta=0;\\
1,&\text{if }\beta=1\text{ and }Y\cap \{0,1\}=\{0\};\\
2,&\text{if }\beta=1\text{ and }Y\cap \{0,1\}=\{1\};\\
3,&\text{if }\beta=1\text{ and }Y\cap \{0,1\}=\{0,1\};\\
4,&\text{if }\beta>1\text{ and }\beta\in Y;\\
0,&\text{otherwise};\\
\end{cases}$$

Consider the stationary set $G:=\{\alpha\in S\cap\tr(Y)\mid g\restriction\alpha=f_\alpha\}$.
Let $\alpha\in G\cap\acc(\kappa)$ be arbitrary. We have $f_\alpha(0)=g(0)=i$, so that $\alpha\in S_i$.
Finally, $g\restriction(\kappa\setminus\{0,1\})$ forms the characteristic function of $Y\setminus\{0,1\}$
and     $g(1)$ encodes $Y\cap\{0,1\}$, so that $Y_\alpha=Y\cap\alpha$.

$(3)\implies(1)$: This is trivial.
\end{proof}
\begin{question} Suppose $X$ (resp.~strongly) $\mathfrak f$-reflects to $S$. Must there exist a partition $\langle S_i\mid i<\kappa\rangle$ of $S$ into stationary sets such that, for all $i<\kappa$,
$X$ (resp.~strongly) $\mathfrak f$-reflects to $S_i$?
\end{question}

\section{Forcing fake reflection}\label{section4}

In this section, we focus on the consistency of \emph{fake reflection}, i.e., $X$ $\mathfrak f$-reflects to $S$ and yet there exists a stationary subset of $X$ that does not reflect (in the classical sense) to $S$.
By the work of Jensen \cite{jensen}, in G\"odel's constructible universe, $\mathsf L$, stationary reflection fails at any non weakly compact cardinal, but, as we will see, filter reflection holds everywhere in $\mathsf L$. We will also show that fake reflection is forceable.

\subsection{A diamond reflecting second-order formulas}

A $\Pi^{1}_{n}$-sentence $\phi$ is a formula of the form $\forall X_1\exists X_2\cdots X_n\varphi$ where $\varphi$ is a first-order sentence over a relational language $\mathcal L$ as follows:
\begin{itemize}
\item $\mathcal L$ has a predicate symbol $\epsilon$ of arity $2$;
\item $\mathcal L$ has a predicate symbols $\mathbb X_i$, $i\leq n$, of arity $m({\mathbb X}_i)$;
\item $\mathcal L$ has infinitely many predicate symbols $(\mathbb A_n)_{n\in \omega}$, each $\mathbb A_m$ is of arity $m(\mathbb A_m)$.
\end{itemize}

\begin{defn} For sets $N$ and $x$, we say that \emph{$N$ sees $x$} iff
$N$ is transitive, p.r.-closed, and $x\cup\{x\}\s N$.
\end{defn}

Suppose that a set $N$ sees an ordinal $\alpha$,
and that $\phi=\forall X_1\exists X_2\cdots\varphi$ is a $\Pi^{1}_{n}$-sentence, where $\varphi$ is a first-order sentence in the above-mentioned language $\mathcal L$.
For every sequence $(A_m)_{m\in\omega}$ such that, for all $m\in\omega$, $A_m\s \alpha^{m(\mathbb A_m)}$,
we write
$$\langle \alpha, \in, (A_{m})_{m\in \omega} \rangle \models_N \phi$$
to express that the two hold:
\begin{enumerate}
\item $(A_{m})_{m\in \omega} \in N$;
\item $N\models (\forall X_1\subseteq \alpha^{m(\mathbb X_1)})(\exists X_2\subseteq \alpha^{m(\mathbb X_2)})\cdots[\langle \alpha, \in, (A_{m})_{m\in \omega}, X_1, X_2,\ldots \rangle\models \varphi]$,
where:
\begin{itemize}
\item $\in$ is the interpretation of $\epsilon$;
\item $X_i$ is the interpretation of $\mathbb X_i$;
\item for all $m\in\omega$,  $A_m$ is the interpretation of $\mathbb A_m$.
\end{itemize}
\end{enumerate}
\begin{conv}\label{conv23}
We write $\alpha^+$ for $|\alpha|^+$,
and write $\langle \alpha,{\in}, (A_{n})_{n\in \omega} \rangle \models \phi$ for
$$\langle \alpha,{\in}, (A_{n})_{n\in \omega} \rangle \models_{H_{\alpha^+}} \phi.$$
\end{conv}

\begin{defn}[Fernandes-Moreno-Rinot, \cite{FMR}]\label{reflectingdiamond}
For a stationary $S\s\kappa$ and a positive integer $n$, $\dl^*_S(\Pi^1_n)$ asserts the existence of a sequence $\vec N=\langle N_\alpha\mid\alpha\in S\rangle$ satisfying the following:

\begin{enumerate}
\item for every $\alpha\in S$, $N_\alpha$ is a set of cardinality $<\kappa$ that sees $\alpha$;
\item for every $X\s\kappa$, there exists a club $C\s\kappa$ such that, for all $\alpha\in C \cap S$, $X\cap\alpha\in N_\alpha$;
\item whenever $\langle \kappa,{\in},(A_m)_{m\in\omega}\rangle\models\phi$,
with $\phi$ a $\Pi^1_n$-sentence,
there are stationarily many $\alpha\in S$ such that $|N_\alpha|=|\alpha|$ and
$\langle \alpha,{\in},(A_m\cap(\alpha^{m(\mathbb A_m)}))_{m\in\omega}\rangle\models_{N_\alpha}\phi$.
\end{enumerate}
\end{defn}
\begin{remark} The principle $\dl^+_S(\Pi^1_n)$ is defined by strengthening Clause~(2)
in the definition of $\dl^*_S(\Pi^1_n)$ to require that $C\cap\alpha$ be in $N_\alpha$, as well.

The Todorcevic-V{\"a}{\"a}n{\"a}nen principle $\diamondsuit^+_S(\Pi^1_n)$ from \cite{TodoVaan} is
obtained by strengthening Clause~(1)
in the definition of $\dl^+_S(\Pi^1_n)$ to require that $|N_\alpha|=\max\{\aleph_0,|\alpha|\}$.
\end{remark}

\begin{lemma}\label{dl-to-fref}
Suppose $S\s\kappa$ is stationary for which $\dl^*_S(\Pi^1_1)$ holds. Then:
\begin{enumerate}
\item $\mathfrak f$-$\refl(\kappa,\kappa,S)$;
\item $\kappa$ $\mathfrak f$-reflects with $\diamondsuit$ to $S$.
\end{enumerate}
\end{lemma}
\begin{proof}
(1) Let $\langle N_\alpha\mid\alpha\in S\rangle$ witness the validity of $\dl^*_S(\Pi^1_1)$.
As in the proof of Lemma~\ref{wc}(1), let $\Phi$ be a $\Pi^1_1$-sentence, such that, for every ordinal $\alpha$,
($\langle \alpha,{\in}\rangle \models \Phi$) iff ($\alpha$ is a regular cardinal).\footnote{Recall Convention~\ref{conv23}.}
Likewise, let $\Psi$ be a $\Pi^1_1$-sentence such that for every ordinal $\alpha$ and every $A\s\alpha\times\alpha$,
$\langle \alpha,{\in},A\rangle \models \Phi$ iff
$$\text{for every }\iota<\alpha, \{\beta<\alpha\mid (\iota,\beta)\in A\}\text{ is stationary in }\alpha.$$

Now, let $S'$ denote the set of all $\alpha\in S$ such that:
\begin{enumerate}
\item[(i)] $\alpha>\omega$;
\item[(ii)] $|N_\alpha|=|\alpha|$;
\item[(iii)] $\langle \alpha,{\in}\rangle\models_{N_\alpha}\Phi$.
\end{enumerate}
For each $\alpha \in S'$, by Clause~(iii), $\mathcal{F}_{\alpha}:=\cub(\alpha)\cap N_\alpha$ is a filter over $\alpha$.

\begin{claim}
Suppose $\langle Y_\iota\mid \iota<\kappa\rangle$ is a sequence of stationary subsets of $\kappa$.
Then there exist stationarily many  $\alpha\in S'$ such that,
for all $\iota<\alpha$, $Y_\iota\cap\alpha\in\mathcal F_\alpha^+$.
\end{claim}
\begin{proof}  Set $A:=\{(\iota,\beta) \mid \iota <\kappa\ \&\ \beta \in Y_\iota\}$.
Clearly, $\langle \kappa,{\in},A\rangle\models\Phi\land\Psi$,
Thus, recalling Clause~(3) of Definition~\ref{reflectingdiamond},
the set $T$ of all $\alpha\in S$ such that $|N_\alpha|=|\alpha|$ and
$\langle \alpha,{\in},A\cap(\alpha\times\alpha)\rangle\models_{N_\alpha}\Phi\land\Psi$ is stationary.
Evidently, every $\alpha\in T\setminus(\omega+1)$ is an element of $S'$ satisfying that, for all $\iota<\alpha$, $Y_\iota\cap\alpha\in\mathcal F_\alpha^+$.
\end{proof}

It follows in particular that $S'$ is stationary.
Finally, recalling Clause~(2) of Definition~\ref{reflectingdiamond},
$\vec{\mathcal{F}}:=\langle\mathcal F_\alpha\mid\alpha\in S'\rangle$ captures clubs.

(2) Continuing the proof of Clause~(1), we see that $\langle N_\alpha\cap\mathcal P(\alpha)\mid \alpha\in S'\rangle$ is a $\diamondsuit^*_{S'}$-sequence,
and $\vec{\mathcal{F}}$ witnesses that $\mathfrak f$-$\refl(\kappa,\kappa,S')$ holds.
The conclusion now follows from Lemma~\ref{diag}.
\end{proof}

In \cite{FMR}, we proved that $\dl^*_S(\Pi^1_2)$ holds in $\mathsf L$  for any stationary subset $S$ of any regular uncountable cardinal $\kappa$. Therefore:
\begin{cor}\label{V=L-fref}
Suppose $\mathsf V=\mathsf L$. Then, for every stationary  $S\s\kappa$,
$\kappa$ $\mathfrak f$-reflects with $\diamondsuit$ to $S$. In particular,
fake reflection holds at any non weakly compact cardinal. \qed
\end{cor}

Furthermore, in \cite{FMR}, we proved that $\dl^*_S(\Pi^1_2)$ follows from a forceable condensation principle called ``Local Club Condensation'' ($\lcc$). In effect,
$\mathfrak f$-reflection is forceable (without assuming any large cardinals).
In this short section, we shall present an alternative and simpler poset for forcing
$\dl^*_S(\Pi^1_2)$ to hold. The idea is to connect the latter with the following strong form of diamond due to Sakai.

\begin{defn}[Sakai, \cite{sakai}] $\diamondsuit^{++}$ asserts the existence of a sequence $\langle K_\alpha\mid \alpha<\omega_1\rangle$ satisfying the following:
\begin{enumerate}
\item for every $\alpha<\omega_1$, $K_\alpha$ is a countable set;
\item for every $X\s\omega_1$, there exists a club $C\s\omega_1$ such that, for all $\alpha\in C$, $C\cap\alpha,X\cap\alpha\in K_\alpha$;
\item the following set is stationary in $[H_{\omega_2}]^\omega$:
$$\{M\in [H_{\omega_2}]^\omega\mid M\cap\omega_1\in\omega_1\land\clps(M,{\in})=(K_{M\cap\omega_1},{\in})\}.$$
\end{enumerate}
\end{defn}

First, we generalize Sakai's principle in the obvious way.
\begin{defn}\label{sakaidiamond}
For a stationary $S\s\kappa$, $\diamondsuit_S^{++}$
asserts the existence of a sequence $\langle K_\alpha\mid \alpha\in S\rangle$ satisfying the following:
\begin{enumerate}
\item for every infinite $\alpha\in S$, $K_\alpha$ is a set of size $|\alpha|$;
\item for every $X\s\kappa$, there exists a club $C\s\kappa$ such that, for all $\alpha\in C \cap S$, $C\cap\alpha,X\cap\alpha\in K_\alpha$;
\item the following set is stationary in $[H_{\kappa^+}]^{<\kappa}$:
$$\{M\in [H_{\kappa^+}]^{<\kappa}\mid M\cap \kappa\in S\ \&\ \clps(M,{\in})= (K_{M\cap\kappa},{\in})\}.$$
\end{enumerate}
\end{defn}
\begin{remark}
For a structure $\mathfrak M$, $\clps(\mathfrak M)$ denotes its Mostowski collapse.
Hereafter, $\zf^{-}$ denotes $\zf$ without the powerset axiom.
\end{remark}

\begin{lemma}\label{lemma49}
For every stationary $S\s\kappa$,
$\diamondsuit_S^{++}$ implies $\diamondsuit^+_S(\Pi^1_2)$.
\end{lemma}
\begin{proof}
Suppose $\langle K_\alpha\mid\alpha\in S\rangle$ is  a $\diamondsuit_S^{++}$-sequence. Define a sequence $\vec{N}=\langle N_\alpha\mid\alpha\in S\rangle$ by letting $N_\alpha$ be the p.r.-closure of $K_\alpha\cup(\alpha+1)$. By the way the sequence $\vec{N}$ was constructed, $N_\alpha$ sees $\alpha$ for all $\alpha\in S$, and by Clause~(1) of Definition~\ref{sakaidiamond}, for every infinite $\alpha\in S$, $|N_\alpha|=|\alpha|$.
In addition, for every $X\s\kappa$, there exists a club $C\s\kappa$ such that
$C\cap\alpha,X\cap\alpha\in K_\alpha\s N_\alpha$  for all $\alpha\in C \cap S$.

Let us show that $\vec{N}$ satisfies Clause~(3) of Definition~\ref{reflectingdiamond} with $n=2$.
To this end, let $\phi=\forall X\exists Y\varphi$ be a $\Pi^1_2$-sentence and $(A_m)_{m\in\omega}$ be such that $\langle \kappa,\in,(A_m)_{m\in\omega}\rangle\models \phi$.
Given an arbitrary club $C\s\kappa$, we consider the following set
$$\mathcal C:=\{ M\prec H_{\kappa^+} \mid M\cap \kappa\in C\ \&\ (A_m)_{m\in\omega}\in M \}.$$
\begin{claim}$\mathcal C$ is a club in $[H_{\kappa^+}]^{<\kappa}$.
\end{claim}
\begin{proof}
By the L\"owenheim--Skolem theorem, for every $B\in[H_{\kappa^+}]^{<\kappa}$, we know that
$$\mathcal M_B=\{M\in [H_{\kappa^+}]^{<\kappa}\mid B\s M\prec H_{\kappa^+}\ \&\ (A_m)_{m\in\omega}\in M \}$$
is a club in $[H_{\kappa^+}]^{<\kappa}$,
so that $\kappa\cap\{M\cap\kappa\mid M\in\mathcal M_B\}$ is a club in $\kappa$.
Consequently, $\mathcal C$ is cofinal in $[H_{\kappa^+}]^{<\kappa}$.

To see that $\mathcal C$ is closed, assume we are given a chain $M_0\subseteq M_1\subseteq\cdots$ of length $\alpha<\kappa$ of elements of $\mathcal C$.
As this is a chain of elementary submodels of $H_{\kappa^+}$ of size smaller than $\kappa$,  $M^*:=\bigcup_{i<\alpha}M_i$ is an elementary submodel of $H_{\kappa^+}$ with $M^*\cap\kappa\in C$, so that $M^*\in\mathcal C$.
\end{proof}
Since $\vec{K}$ a $\diamondsuit_S^{++}$-sequence,
we may now pick $M$ in the following intersection
$$\mathcal C\cap \{M\in [H_{\kappa^+}]^{<\kappa}\mid M\cap \kappa\in S\ \&\ \clps (M,{\in})=(K_{M\cap\kappa},{\in})\}.$$
So, $M\prec H_{\kappa^+}$,
$(A_m)_{m\in\omega}\in M$,
$\alpha:=M\cap\kappa$ is in $S\cap C$,
and $\clps (M,{\in})=(K_\alpha,{\in})$.

As $M\cap(\kappa+1)=\alpha\cup\{\kappa\}$, $\alpha\cup\{\alpha\}$ is a subset of the collapse of $M$,
so that $K_\alpha$ sees $\alpha$ and $N_\alpha=K_\alpha$.
Let $\pi:M\rightarrow N_\alpha$ denote the transitive collapsing map.
Note that \begin{itemize}
\item[(i)] $\pi\restriction\alpha=\id_\alpha$,
\item[(ii)] $\pi(\kappa)=\alpha$, and
\item[(iii)]  $\pi(A_m)=A_m\cap\alpha$, for all $m\in\omega$.
\end{itemize}

Since $\langle \kappa,\in,(A_m)_{m\in\omega}\rangle\models \phi$, by definition, $\langle \kappa,\in,(A_m)_{m\in\omega}\rangle\models_{H_{\kappa^+}} \forall X\exists Y\varphi$.
That is,  $$H_{\kappa^+}\models ``\forall X\subseteq \kappa^{m(\mathbb{X})}\ \exists Y\subseteq \kappa^{m(\mathbb{Y})}\ \langle \kappa,\in,(A_m)_{m\in\omega}\rangle\models \varphi".$$
By elementarity and the fact that $``\forall X\subseteq \kappa^{m(\mathbb{X})}\exists Y\subseteq \kappa^{m(\mathbb{Y})}(\langle \kappa,\in,(A_m)_{m\in\omega}\rangle\models \varphi)" $  is equivalent to
\begin{gather*}
\forall X ((\forall x  (x \in X \rightarrow x \in \kappa^{m(\mathbb X)}))\rightarrow\hspace{135pt}\\\hspace{100pt}(\exists Y((\forall y (y \in Y \rightarrow y \in \kappa^{m(\mathbb Y)}))\land\langle \kappa,\in,(A_m)_{m\in\omega}\rangle\models \varphi ))),
\end{gather*}
which is a first-order formula in the parameters $m(\mathbb X)$, $m(\mathbb Y)$, $\kappa$, $\langle \kappa, \in, \vec{A} \rangle$ and  $\varphi$, we have
$$M\models ``\forall X\subseteq \kappa^{m(\mathbb X)}\exists Y\subseteq \kappa^{m(\mathbb Y)}(\langle \kappa,\in,(A_m)_{m\in\omega}\rangle\models \varphi)".$$
Since $\pi$ is an elementary embedding,
$$\pi[M]\models ``\forall X\subseteq \pi(\kappa^{m(\mathbb X)}) \exists Y\subseteq \pi(\kappa^{m(\mathbb Y)})(\langle \pi(\kappa),\in,(\pi(A_m))_{m\in\omega}\rangle\models \varphi)".$$

By the properties (i),(ii) and (iii) of $\pi$ it follows that
$$N_\alpha\models ``\forall X\subseteq \alpha^{m(\mathbb X)} \exists Y\subseteq \alpha^{m(\mathbb Y)}(\langle \alpha,\in,(A_m\cap(\alpha^{m(\mathbb A_m)}))_{m\in\omega}\rangle\models \varphi)".$$
We conclude $\langle \alpha,{\in},(A_n\cap(\alpha^{m(\mathbb A_n)}))_{n\in\omega}\rangle\models_{N_\alpha}\phi$, as sought.
\end{proof}
\begin{remark} An obvious tweaking of the above proof shows that
$\langle N_\alpha\mid\alpha\in S\rangle$ in fact witnesses
$\diamondsuit^+_S(\Pi^1_n)$ for every positive integer $n$.
\end{remark}

The following answers a question of Thilo Weinert:\footnote{Communicated in person to the third author in 2017.}

\begin{cor}\label{PlusNotPlusPlus}
It is consistent that $\diamondsuit^+_S$ holds, but $\diamondsuit_S^{++}$ fails.
\end{cor}
\begin{proof} In \cite[\S4]{FMR}, we identified a model in which $\diamondsuit_S^+$ holds
for $S:=\omega_2\cap\cof(\omega)$ but  $\dl^*_S(\Pi^1_2)$ fails. By Lemma~\ref{lemma49},
$\diamondsuit^{++}_S$ fails in this model, as well.
\end{proof}

In \cite[Definition~3.1]{sakai}, Sakai presented a poset for forcing $\diamondsuit^{++}$ to hold.
The following is an obvious generalization (and a minor simplification) of Sakai's poset.

\begin{defn}\label{sakai-forcing}
Let $\mathbb S$ be the poset of all pairs $(k,\mathcal{B})$ with the following properties:
\begin{enumerate}
\item $k$ is a function such that \text{dom}$(k)<\kappa$;
\item for each $\alpha\in \text{dom}(k), k(\alpha)$ is a transitive model of $\zf^-$ of size $\le\max\{\aleph_0,|\alpha|\}$, with $k\restriction\alpha\in k(\alpha)$;
\item $\mathcal{B}$ is a subset of $\mathcal P(\kappa)$ of size $\le\dom(k)$;
\end{enumerate}
$(k',\mathcal{B}')\leq(k,\mathcal{B})$ in $\mathbb P$ if the following holds:
\begin{enumerate}
\item[(i)] $k'\supseteq k$, and $\mathcal{B}'\supseteq \mathcal{B}$;
\item[(ii)] for any $B\in \mathcal{B}$ and any $\alpha\in \text{dom}(k')\setminus \text{dom}(k)$, $B\cap\alpha\in k'(\alpha)$.
\end{enumerate}
\end{defn}

It is clear that $\mathbb S$ is ${<}\kappa$-closed.
Also, since we assume $\kappa^{<\kappa}=\kappa$, $\mathbb S$ has the $\kappa^+$-cc.
Finally, Sakai's proof of \cite[Lemma~3.4]{sakai} makes clear that the following holds.
\begin{prop} For every stationary $S\s\kappa$, $V^{\mathbb S}\models \diamondsuit^{++}_S$.\qed
\end{prop}

Note that while Sakai's forcing is considerably simpler than the poset to force $\lcc$ to hold,
it only yields ``$\kappa$ $\mathfrak f$-reflects to $S$'' for stationary subsets $S\s\kappa$ from the ground model,
whereas, $\lcc$ imply that $\kappa$ $\mathfrak f$-reflects to $S$ for any stationary $S\s\kappa$.

\begin{remark} For stationary subsets $X,S$ of $\kappa$,
if $X$ $\mathfrak f$-reflects to $S$,
then for every notion of forcing $\mathbb P$ of size ${<}\kappa$,
$V^{\mathbb P}\models X\ \mathfrak f\text{-reflects to }S$.
It takes a little more effort, but it can be shown that
if $X$ $\mathfrak f$-reflects with $\diamondsuit$ to $S$,
then for every notion of forcing $\mathbb P$ of size ${<}\kappa$,
$V^{\mathbb P}\models X\ \mathfrak f\text{-reflects with }\diamondsuit\text{ to }S$.
\end{remark}

In this section and in the previous one, we have collected a long list of sufficient conditions for filter reflection to hold.
We have seen it is compatible with large cardinals, strong forcing axioms, but also with inner models like $\mathsf{L}$, in which anti-reflection principles like $\square_\lambda$ hold.
This suggests it is not trivial to destroy filter reflection. The next section is dedicated to demonstrating it is nevertheless possible.

\section{Killing fake reflection}\label{sectionFFR}

\begin{defn}\label{appideal}    Let $ X\s \kappa$.
We define a collection $I[\kappa-X]$, as follows.

A set $Y$ is in $I[\kappa - X ]$ iff $Y\s\kappa$ and there exists a sequence
$\langle a_\beta\mid\beta<\kappa\rangle$ of elements of $[\kappa]^{<\kappa}$ along with a club $C\s\kappa$
such that, for every $\delta \in Y\cap C$, there is a cofinal subset $A\s\delta$ of order-type $\cf(\delta)$ such that
\begin{enumerate}
\item $ \{A \cap \gamma\mid \gamma< \delta \} \s \{a_{\beta} \mid \beta < \delta\}$, and
\item $\acc^+(A)\cap X=\emptyset$.
\end{enumerate}
\end{defn}
\begin{remark}
Note that $I[\kappa-X]$ is an ideal, and that $X\s X'$ entails $I[\kappa-X]\supseteq I[\kappa-X']$.
Shelah's \emph{approachability ideal} $I[\kappa]$ is equal to $I[\kappa-\emptyset]\restriction\sing$ (cf.~\cite{Todd}).
In particular, for every $\mu\in\reg(\kappa)$, $I[\kappa]\restriction\cof(\mu)$ coincides with $I[\kappa-\emptyset]\restriction\cof(\mu)$.
\end{remark}

\begin{fact}[folklore]\label{Cohen}
Every separative ${<}\kappa$-closed notion of forcing of size $\kappa$ is forcing equivalent to $\add(\kappa,1)$.
\end{fact}

\begin{thm}\label{CohenEffect2} Suppose $X,S$ are disjoint stationary subsets of $\kappa$,
with $S \in\allowbreak{I[\kappa - X]}$.
For every  $\vec{\mathcal F}=\langle\mathcal F_\alpha\mid \alpha\in S\rangle$, $V^{\add(\kappa,1)}\models X\text{ does not }\vec{\mathcal F}\text{-reflect to }S$.
\end{thm}
\begin{proof} Towards a contradiction, suppose that $\vec{\mathcal F}$ is a counterexample.
As $\add(\kappa,1)$ is almost homogeneous and $X,S,\vec{\mathcal F}$ live in the ground model,
it follows that, in fact, $V^{\add(\kappa,1)}\models X~\vec{\mathcal F}\text{-reflects to }S$.

Let $R$ denote the set of all pairs $(p,q)\in 2^{<\kappa}\times 2^{<\kappa}$ such that:
\begin{itemize}
\item $\dom(p)=\dom(q)$ is in $\nacc(\kappa)$;
\item $\{\alpha\in \dom(p)\mid p(\alpha)=q(\alpha)=1\}$ is disjoint from $X$;
\item $\{\alpha\in \dom(q)\mid q(\alpha)=1\}$ is a closed set of ordinals.
\end{itemize}

We let $\mathbb R:=(R,\le)$ where $(p',q')\le (p,q)$ iff $p'\supseteq p$ and $q'\supseteq q$.
\begin{claim}$\mathbb R$ is ${<}\kappa$-closed.
\end{claim}
\begin{proof} Given $\theta\in\acc(\kappa)$ and a strictly decreasing sequence $\langle (p_i,q_i)\mid i<\theta\rangle$
of conditions in $\mathbb R$, let $p:=(\bigcup_{i<\theta}p_i){}^\curvearrowright0$
and $q:=(\bigcup_{i<\theta}q_i){}^\curvearrowright1$.
Clearly, $(p,q)$ is a legitimate condition extending $(p_i,q_i)$ for all $i<\theta$.
\end{proof}

It thus follows from Fact~\ref{Cohen} that $\mathbb R$ is forcing equivalent to $\add(\kappa,1)$.
Also, let $P:=\{p\mid \exists q~(p,q)\in R\}$. It is easy to see that $\mathbb P:=(P,\supseteq)$ is ${<}\kappa$-closed,
so that $\mathbb P$ is, as well, forcing equivalent to $\add(\kappa,1)$.
Next, let $G$ be $\mathbb R$-generic over $V$.
Let $G_0$ denote the projection of $G$ to the first coordinate,
so that $G_0$ is $\mathbb P$-generic over $V$.
In $V[G_0]$, let $Q:=\{q\in{}2^{<\kappa}\mid \exists p\in G_0~(p,q)\in R\}$.
Clearly, $\mathbb Q:=(Q,\supseteq)$ is isomorphic to the quotient forcing $\mathbb R/G_0$.
It follows that, in $V[G]$, we may read a $\mathbb Q$-generic set $G_1$ over $V[G_0]$
such that, in particular, $V[G]=V[G_0][G_1]$.

Denote $\eta:=\bigcup G_0$ and let $Y:=\{\alpha\in X\mid \eta(\alpha)=1\}$.

\begin{claim}\label{cohen2} In $V[G_0]$, $Y$ is stationary.
\end{claim}
\begin{proof}
We run a density argument for $\mathbb P$ in $V$.
Let $\dot Y$ be the $\mathbb P$-name for $Y$, that is,
$$\dot Y:=\{ (\check\alpha,p)\mid p\in P, \alpha\in X\cap\dom(p), p(\alpha)=1\}.$$
Let $p$ be an arbitrary condition that $\mathbb P$-forces that some $\dot D$ is a $\mathbb P$-name for a club in $\kappa$;
we shall find $p^\bullet\supseteq p $ such that $p^\bullet\forces_{\mathbb{P}}\dot{D}\cap\dot Y\neq\emptyset$.

Recursively define a sequence $\langle (p_i,\alpha_i)\mid i<\kappa\rangle$ as follows:
\begin{itemize}
\item[$\br$] Let $(p_0,\alpha_0)$ be such that $p_0\supseteq p$ and $p_0\forces_{\mathbb P}\check\alpha_0\in \dot D$.

\item[$\br$] Suppose that $i<\kappa$ for which $\langle (p_j,\alpha_{j}) \mid j \leq i \rangle $ has already been defined.
Set $\varepsilon_i:=\max\{\alpha_i,\dom(p_i)\} +1$.
Then pick $p_{i+1}\supseteq p_i$ and $\alpha_{i+1}<\kappa$
such that $\varepsilon_i\in \dom(p_{i+1})$ and $p_{i+1}\forces_{\mathbb P}\check\alpha_{i+1}\in \dot D\setminus\check\varepsilon_i$.

\item[$\br$] Suppose that $i\in\acc(\kappa)$ and that $\langle (p_j,\alpha_{j}) \mid j < i \rangle $ has already been defined.
Evidently, $$\sup_{j<i}\varepsilon_j=\sup_{j<i}(\dom(p_j))=\sup_{j<i}\alpha_j,$$
so we let $\alpha_i$ denote the above common value.

Finally, set $p_i:=(\bigcup_{j<i}p_j){}^\curvearrowright 1$,
so that $p_i$ is a legitimate condition satisfying $\dom(p_i)=\alpha_i+1$ and $p_i(\alpha_i)=1$.
\end{itemize}
This completes the recursive construction.
Evidently, $E:=\{ \alpha_i\mid i<\kappa\}$ is a club,
so as  $X$ is stationary,
we may pick $\beta\in X$ such that $\alpha_\beta=\beta$.
Then $p_\beta\forces_{\mathbb{P}}\check\beta\in \dot D\cap \check X$,
so that, from $p_\beta(\beta)=1$, we infer that $p_\beta\forces_{\mathbb{P}}\dot{D}\cap\dot Y\neq\emptyset$.
\end{proof}

Work in $V[G_0]$. Since $X$ $\vec{\mathcal F}$-reflects to $S$,
$T:=\{\alpha\in S\mid Y\cap\alpha\in\mathcal F_\alpha^+\}$ is stationary.

\begin{claim}\label{cohen3} In $V[G_0][G_1]$, $T$ is stationary.
\end{claim}
\begin{proof}
Fix $\vec{a}$, $C$ in $V$ that witness together that $S$ is in $I[\kappa - X]$.
As $\mathbb P$ is cofinality-preserving, in $V[G_0]$, the above two still witness together that $S$ is in $I[\kappa-X]$.
Work in $V[G_0]$.
As $T$ is a subset of $S$,
$\vec{a}$, $C$ also witness together that $T$ is in $I[\kappa-X]$.

We now run a density argument for $\mathbb Q$ in $V[G_0]$.
Let $q$ be an arbitrary condition that $\mathbb Q$-forces that some $\dot D$ is a $\mathbb Q$-name for a club in $\kappa$;
we shall find $q^\bullet \supseteq q $ such that $q^\bullet\forces_{\mathbb{Q}}\dot{D}\cap\check T\neq\emptyset$.

Fix a large enough regular cardinal $\Theta$ and some well-ordering $<_\Theta$ of $H_\Theta$.
By Claim~\ref{cohen2}, $T$ is stationary,
so we may find an elementary submodel $N\prec\left(H_\Theta,<_\Theta\right)$ such that $\vec a,C,\mathbb Q,q,\dot D\in N$ and $\delta:=N\cap \kappa$ is in $T$.

As $C\in N$, we altogether have $\delta\in C\cap T$.
Thus, we may pick a cofinal subset $ A\s \delta $ with $\otp(A)=\cf(\delta)$ and $\acc^+(A) \cap X = \emptyset $ such that:
$$\{ A \cap \gamma\mid \gamma< \delta\}\s \{a_{\beta} \mid \beta < \delta \}.$$
In particular, any proper initial segment of $A$ is in $N$.

Let $\langle \delta_i \mid {i} <\cf(\delta)\rangle$ be the increasing enumeration of $A$.
For every initial segment $a$ of $A$,
we recursively define the following sequence $\langle (q_{i},\alpha_{i})\mid{i}\le\sigma(a)\rangle$,
where $\sigma(a)$ will the length of the recursion (see the second case below).

\begin{itemize}
\item[$\br$] Let $q_{0}$ be the $<_\Theta$-least condition in $\mathbb{Q}$ extending $q$ for which there is $\alpha<\kappa$ such that $q_{0} \forces_{\mathbb Q}\check{\alpha} \in \dot D$.
Now, let  $\alpha_{0}$ be the $<_\Theta$-least ordinal $\alpha$ such that $q_{0}\forces_{\mathbb Q}\check{\alpha} \in \dot D$.
\item[$\br$]  Suppose that $\langle (q_{j},\alpha_{j}) \mid j \leq i \rangle$ has already been defined. If $a\setminus\max\{\alpha_i,\allowbreak\dom(q_{i}),\delta_{i}\}$
is empty, then we terminate the recursion, and set  $\sigma(a):={i}$. Otherwise, let
$\varepsilon_i$ be the $<_\Theta$-least element of $a\setminus\max\{\alpha_i,\dom(q_{i}),\delta_{i}\}$,
and then let  $q_{{i}+1}$ be the $<_\Theta$-least condition in $\mathbb{Q}$ extending $q_i$ satisfying $\varepsilon_{i}\in\dom(q_{{i}+1})$ and
satisfying that there is $\alpha <\kappa$ such that $q_{{i}+1} \forces_{\mathbb Q}\check{\alpha} \in \dot D \setminus \varepsilon_i$.
Now, let $\alpha_{{i}+1}$ be the $<_\Theta$-least ordinal $\alpha$ such that $q_{{i}+1} \forces_{\mathbb Q}\check{\alpha} \in \dot D \setminus \varepsilon_i$.
\item[$\br$] Suppose that ${i}$ is a limit ordinal and that $\langle (q_{j},\alpha_{j}) \mid j<{i}\rangle$ has already been defined.
Evidently, $$\sup_{j<i}\varepsilon_j=\sup_{j<i}(\dom(q_j))=\sup_{j<i}\alpha_j,$$
so we let $\alpha_i$ denote the above common value. As $\{ \varepsilon_j\mid j<i\}\s a\s A$ and as $\acc^+(A)\cap X=\emptyset$, we infer that $\alpha_i\notin  X$.
So, $q_i:=(\bigcup_{j<{i}}q_{j}){}^\curvearrowright 1$ is a legitimate condition
satisfying $\dom(q_i)=\alpha_i+1$ and $q_i(\alpha_i)=1$.
\end{itemize}

This completes the recursive construction.
Since every proper initial segment of $A$ is in $N$,
for every $\gamma<\cf(\delta)$, $\langle (q_i,\alpha_i)\mid i\le\sigma(A\cap\gamma)\rangle$ is in $N$,
so that $\sigma(A)=\cf(\delta)$ and $\alpha_{\cf(\delta)}=\delta$. Altogether, $q_{\cf(\delta)}\forces_{\mathbb Q}\check\delta\in\dot D$.
Recalling that $\delta\in T$, our proof is complete.
\end{proof}
\begin{claim}\label{cohen4} In $V[G_0][G_1]$, $Y$ is non-stationary.
\end{claim}
\begin{proof}
It is clear that $C:=\{ \alpha<\kappa\mid \exists q\in G_1(q(\alpha)=1)\}$ is a closed subset of $\kappa$
which is disjoint from $Y$. Thus, we are left with proving that $C$ is unbounded in $\kappa$.
To this end, we run a density argument for $\mathbb Q$ in $V[G_0]$.
For every condition $q$ in $\mathbb Q$,
find $\delta\in S$ above $\dom(q)$,
and then define $q^\bullet:\delta+1\rightarrow 2$ via:
$$q^\bullet(\alpha):=\begin{cases}
q(\alpha),&\text{if }\alpha\in\dom(q);\\
1,&\text{if }\alpha=\delta;\\
0,&\text{otherwise}.\end{cases}$$

As $X\cap S=\emptyset$, $q^\bullet$ is a legitimate condition extending $q$,
and, in addition, $\{\alpha<\kappa\mid q^\bullet(\alpha)=1\}$ is a proper end-extension of
$\{\alpha<\kappa\mid q(\alpha)=1\}$.
\end{proof}

Work in $V[G_0][G_1]$. Fix a club $C$ disjoint from $Y$.
Since $X$ $\vec{\mathcal F}$-reflects to $S$, in particular, $\vec{\mathcal F}$ capture clubs,
so that $\{\alpha\in S\mid C\cap\alpha\notin\mathcal F_\alpha\}$ is non-stationary.
Recalling that $T$ is stationary, we now fix $\alpha\in T$ such that $C\cap\alpha\in\mathcal F_\alpha$.
By definition of $T$, we also have $Y\cap\alpha\in\mathcal F_\alpha^+$, so that $(C\cap\alpha)\cap(Y\cap\alpha)$ is nonempty,
contradicting the fact that $C$ is disjoint from $Y$.
\end{proof}

\begin{cor}\label{corollary4122} Suppose $X,S$ are disjoint stationary subsets of $\kappa$,
with $S \in I[\kappa - X]$.
After forcing with $\add(\kappa,\kappa^{+})$, $X$ does not $\mathfrak f$-reflect to $S$.
\end{cor}
\begin{proof} Let $G$ be $\add(\kappa,\kappa^+)$-generic over $V$.
For every $\iota\le\kappa^+$, let $G_\iota$ denote the projection of $G$ into the $\iota^{th}$ stage.

Work in $V[G_{\kappa^+}]$.
Towards a contradiction, suppose that $S'$ is a stationary subset of $S$,
and $\vec{\mathcal F}=\langle \mathcal F_\alpha\mid\alpha\in S'\rangle$ is a sequence such that $X$ $\vec{\mathcal F}$-reflects to $S'$.
As $\add(\kappa,\kappa^+)$ does not add bounded subsets of $\kappa$, $\vec{\mathcal F}\s(H_\kappa)^V$.
In addition, $\add(\kappa,\kappa^+)$ has the $\kappa^{+}$-cc, so that, altogether, $\vec{\mathcal F}$ admits a nice name of size $\kappa$.
It follows that we may find a large enough $\iota<\kappa^+$ such that
$\vec{\mathcal F}$ is in $V[G_\iota]$.
Now, by Theorem~\ref{CohenEffect2},
$V[G_{\iota+1}]\models``X\text{ does not }\vec{\mathcal F}\text{-reflect to }S'"$.
Recalling that $V[G_{\kappa^+}]\models``X~\vec{\mathcal F}\text{-reflect to }S'"$,
it must be the case that there exists a stationary subset of $\kappa$ in $V[G_{\iota+1}]$ that ceases to be stationary in $V[G_{\kappa^+}]$.
However, the quotient forcing $\add(\kappa,\kappa^+)/G_{\iota+1}$ is isomorphic to $\add(\kappa,\kappa^+)$
and the latter preserves stationary subsets of $\kappa$. This is a contradiction.
\end{proof}

\begin{lemma}\label{lemma515} Suppose that $\kappa$ is strongly inaccessible or $\kappa=\lambda^+$ with $\lambda^{<\lambda}=\lambda$.
For every stationary $X,Y\s\kappa$ such that $\tr(X)\cap Y$ is non-stationary, $Y\in I[\kappa-X]$.
\end{lemma}
\begin{proof} Let $\theta:=\sup(\reg(\kappa))$, so that $\kappa\in\{\theta,\theta^+\}$,
and, for every $\gamma<\kappa$, $|[\gamma]^{<\theta}|<\kappa$.
Let $\langle a_\beta\mid \beta<\kappa\rangle$
be some enumeration of $[\kappa]^{<\theta}$,
and then define a function $f:\kappa\rightarrow\kappa$ via:
$$f(\gamma):=\min\{\tau<\kappa\mid \mathcal [\gamma]^{<\theta}\s\{a_\beta\mid \beta<\tau\}\}.$$

Now, fix arbitrary $X,Y\s\kappa$ for which $\tr(X)\cap Y$ is non-stationary.
To see that $Y\in I[\kappa-X]$,
fix a subclub $C$ of $\{\delta<\kappa \mid f[\delta]\s\delta\}$ disjoint from
$\tr(X)\cap Y$. Let $\delta\in Y\cap C$ be arbitrary.
Fix a club $A$ in $\delta$ of order-type $\cf(\delta)$ such that $A\cap X=\emptyset$.
By definition of $\theta$, we have $\cf(\delta)\le\theta$, so that, for every $\gamma<\delta$,
we have $\otp(A\cap\gamma)<\theta$ and $f(\gamma)<\delta$,
and hence there must exist some $\beta<\delta$ with $A\cap\gamma=a_\beta$.
\end{proof}

We are now ready to derive Theorem~C:

\begin{cor}
If $\kappa$ is strongly inaccessible, then in the forcing extension by $\add(\kappa,\kappa^+)$, for all two disjoint stationary subsets $X,S$ of $\kappa$,
the following are equivalent:
\begin{enumerate}
\item $X$ $\mathfrak f$-reflects to $S$;
\item every stationary subset of $X$ reflects in $S$.
\end{enumerate}
\end{cor}
\begin{proof}
The implication $(2)\implies(1)$ holds true in any model,
since if $X$ reflects in $S$, then
$S':=S\setminus\cof(\omega)$ must be stationary,
and $X$ $\langle \mathcal \cub(\alpha)\mid \alpha\in S'\rangle$-reflects to $S'$.
In particular, $X$ $\mathfrak f$-reflects to $S$.
Thus, we shall focus on the other implication.

Let $G$ be $\add(\kappa,\kappa^+)$-generic over $V$.
For every $\iota\le\kappa^+$, we let $G_\iota$ denote the projection of $G$ into the $\iota^{th}$ stage.
Work in $V[G]$.
We verify that  $\neg(2)\implies\neg(1)$.

Suppose that $X$ and $S$ are disjoint stationary subsets of $\kappa$ such that $X$ admits a stationary subset $Z\subseteq X$ that does not reflect in $S$.
As $Z$ and $S$ are elements of $H_{\kappa^+}$
and as $\add(\kappa,\kappa^+)$ has the $\kappa^{+}$-cc, we may find a large enough $\iota<\kappa^+$
such that $Z$ and $S$ are in $V[G_\iota]$.
As the quotient forcing $\add(\kappa,\kappa^+)/G_{\iota}$ is isomorphic to $\add(\kappa,\kappa^+)$
and the latter does not add bounded subsets of $\kappa$ and does preserve stationary subsets of $\kappa$,
also, in $V[G_\iota]$, $Z$ does not reflect in $S$.
Now, by Lemma~\ref{lemma515}, $S\in I[\kappa-Z]$.
So, since $\add(\kappa,\kappa^+)/G_{\iota}$ is isomorphic to $\add(\kappa,\kappa^+)$,
Corollary~\ref{corollary4122} implies that
$Z$ does not $\mathfrak f$-reflect to $S$  in $V[G]$.
But $Z\s X$, contradicting Monotonicity Lemma~\ref{monotonicity}.
\end{proof}

\begin{defn} Let $\vec C=\langle C_\alpha\mid \alpha\in\Gamma\rangle$  be some sequence,
with $\Gamma\s\ord$.
\begin{itemize}
\item  $\vec C$  is said to be a \emph{$C$-sequence over $\Gamma$}
iff, for every $\alpha\in\Gamma$,
$C_\alpha$ is a closed subset of $\alpha$ with $\sup(C_\alpha)=\sup(\alpha)$;
\item $\vec C$ is said to be \emph{coherent} iff, for all $\alpha\in\Gamma$ and $\bar\alpha\in\acc(C_\alpha)$, $\bar\alpha\in\Gamma$ and $C_{\bar\alpha}=C_\alpha\cap\bar\alpha$;
\item $\vec C$ is said to be \emph{regressive} iff $\otp(C_\alpha)<\alpha$ for all $\alpha\in\Gamma$.
\end{itemize}
\end{defn}
\begin{remark}
\begin{enumerate}
\item Jensen proved \cite{jensen} that if $V=L$, then there exists a coherent regressive $C$-sequence over $\sing$ (the class of infinite singular ordinals).
\item Jensen's principle $\square_\lambda$ is equivalent to the assertion that
there exists a coherent regressive $C$-sequence over a club in $\lambda^+$.
\item By \cite{Sh:351}, for every regular uncountable cardinal $\lambda$,
there exists a sequence $\langle \Gamma_i\mid i<\lambda\rangle$ such that $\bigcup_{i<\lambda}\Gamma_i=\acc(\lambda^+\setminus\lambda)\cap\cof({<}\lambda)$ and, for all $i<\lambda$, there exists a coherent regressive $C$-sequence over $\Gamma_i$.
\item By \cite{Sakaisquare}, $\mm$ implies the existence of a coherent regressive $C$-sequence over some $\Gamma\s\omega_2$ for which $\Gamma\cap\cof(\omega_1)$ is stationary.
\end{enumerate}
\end{remark}
\begin{lemma}\label{usingsquare}
Let $\Gamma\s\sing(\kappa)$ be stationary. If there exists a coherent regressive
$C$-sequence over $\Gamma$,
then, for every stationary $X\s\Gamma$, there exists a stationary $Z\s X$ with $\Gamma\in I[\kappa-Z]$.
\end{lemma}
\begin{proof} Suppose that $\langle C_\alpha\mid\alpha\in\Gamma\rangle$ is a coherent regressive $C$-sequence.
For every $\epsilon<\kappa$, let $\Gamma_\epsilon:=\{\alpha\in\Gamma\mid \otp(C_\alpha)=\epsilon\}$.
By Fodor's lemma, for every stationary $X\s\Gamma$, there must exist some $\epsilon<\kappa$ such that $X\cap\Gamma_\epsilon$ is stationary.
Thus, we shall focus on proving that $\Gamma\in I[\kappa-\Gamma_\epsilon]$ for all $\epsilon<\kappa$.

For all $\eta\in \kappa\cap\cof({>}\omega)$ and $\epsilon<\kappa$, fix a subclub $d_{\eta,\epsilon}$ of $\acc(\eta)$ of order-type $\cf(\eta)$
such that $\epsilon\notin d_{\eta,\epsilon}$;
then, for every $\alpha\in\Gamma$, let
$$C_\alpha^{\eta,\epsilon}:=\{ \zeta\in C_\alpha\mid \otp(C_\alpha\cap\zeta)\in d_{\eta,\epsilon}\}.$$
To help the reader digest the above definition, we mention that each such a set $C_\alpha^{\eta,\epsilon}$ is a closed (possibly empty) subset of $\acc(C_\alpha)$ of order-type $\le\otp(d_{\eta,\epsilon})$.

Next, let $\vec a=\langle a_\beta\mid\beta<\kappa\rangle$ be some enumeration of
$$[\kappa]^{<\omega}\cup\{ C_\alpha^{\eta,\epsilon}\mid \alpha\in\Gamma, \eta\in\kappa\cap\cof({>}\omega), \epsilon<\kappa\}.$$
Fix a club $D$ in $\kappa$ such that, for all $\delta<\kappa$:
\begin{itemize}
\item $[\delta]^{<\omega}\s \{ a_\beta\mid \beta<\delta\}$, and
\item $\{ C_\alpha^{\eta,\epsilon}\mid \alpha\in\Gamma\cap\delta, \eta\in\delta\cap\cof({>}\omega), \epsilon<\delta\}\s \{ a_\beta\mid \beta<\delta\}$.
\end{itemize}

Let $\epsilon<\kappa$. We claim that $\vec a$ and $D\setminus(\epsilon+1)$ witness together that $\Gamma\in I[\kappa-\Gamma_\epsilon]$.
To this end, let $\delta\in\Gamma\cap D\setminus(\epsilon+1)$ be arbitrary. There are two cases to consider:

$\br$ If $\cf(\delta)=\omega$, then let $A$ be an arbitrary cofinal subset of $\delta$ of order-type $\omega$.
Clearly, $\acc^+(A)=\emptyset$. In addition, by $\delta\in D$, any proper initial segment of $A$ is indeed listed in $\{a_\beta\mid \beta<\delta\}$.

$\br$ If $\cf(\delta)>\omega$, then let $\eta:=\otp(C_\delta)$, so that $\eta<\delta$.
As $C_\delta$ is a club in $\delta$,
$\cf(\eta)=\cf(\delta)$.
As $d_{\eta,\epsilon}$ is a subclub of $\acc(\eta)$ of order-type $\cf(\eta)=\cf(\delta)$,
$A:=C_\delta^{\eta,\epsilon}$ is a subclub of $\acc(C_\delta)$ of order-type $\cf(\delta)$.
Now, if $\alpha\in A\cap \Gamma_\epsilon$, then $\alpha\in\acc(C_\delta)$, so that $\alpha\in\Gamma$ and $C_\delta\cap\alpha=C_\alpha$,
and also $\alpha\in \Gamma_\epsilon$ so that $\otp(C_\delta\cap\alpha)=\otp(C_\alpha)=\epsilon$.
Recalling that $\alpha\in A=C_\delta^{\eta,\epsilon}=\{ \zeta\in C_\delta\mid \otp(C_\delta\cap\zeta)\in d_{\eta,\epsilon}\}$, this means that $\epsilon\in d_{\eta,\epsilon}$,
contradicting the choice of $d_{\eta,\epsilon}$. It follows in particular that $\acc^+(A)\cap \Gamma_\epsilon=\emptyset$.

Finally, let $\gamma<\delta$, and we shall show that $A\cap\gamma\in\{a_\beta\mid \beta<\delta\}$.
Put $\alpha:=\min(A\setminus\gamma)$, so that $A\cap\gamma=A\cap\alpha$. As observed earlier, the fact that $\alpha\in A$
entails $C_\alpha=C_\delta\cap\alpha$.
But then $C_\alpha^{\eta,\epsilon}=C_\delta^{\eta,\epsilon}\cap\alpha=A\cap\gamma$. Recalling that $\delta\in D$
and $\max\{\alpha,\eta,\epsilon\}<\delta$,
we infer that the proper initial segment $A\cap\gamma$ is indeed listed in $\{a_\beta\mid\beta<\delta\}$.
\end{proof}

\begin{cor}\label{cor5182}  Suppose that there exists a coherent regressive
$C$-sequence over a stationary subset $\Gamma$ of $\kappa$.
After forcing with $\add(\kappa,\kappa^{+})$,
for all two disjoint stationary subsets $X,S$ of $\Gamma$,
$X$ does not $\mathfrak f$-reflect to $S$.
\end{cor}
\begin{proof}
Let $G$ be  $\add(\kappa,\kappa^+)$-generic over $V$.
As in the proof of Corollary~\ref{corollary4122},
for every $\iota\le\kappa^+$, we let $G_\iota$ denote the projection of $G$ into the $\iota^{th}$ stage.

Work in $V[G]$. Suppose that $X$ and $S$ are disjoint stationary subsets of $\Gamma$.
As $X$ and $S$ are elements of $H_{\kappa^+}$
and as $\add(\kappa,\kappa^+)$ has the $\kappa^{+}$-cc, we may find a large enough $\iota<\kappa^+$
such that $X$ and $S$ are in $V[G_\iota]$.
Now, in $V$, and hence also in $V[G_\iota]$, there exists a coherent regressive $C$-sequence over $\Gamma$.
So, by Lemma~\ref{usingsquare}, there is a stationary $Z\subseteq X$ such that $Y\in I[\kappa-Z]$.
Finally, since $\add(\kappa,\kappa^+)/G_{\iota}$ is isomorphic to $\add(\kappa,\kappa^+)$,
Corollary~\ref{corollary4122} implies that, in $V[G]$, $Z$ does not $\mathfrak f$-reflect to $S$.
As $Z\s X$, it follows from Monotonicity Lemma~\ref{monotonicity} that,  in $V[G]$, $X$ does not $\mathfrak f$-reflect to $S$.
\end{proof}

We can now derive Theorem~B:

\begin{cor}[Dense non-reflection]\label{corollary514} There exists a cofinality-preserving forcing extension
in which:
\begin{enumerate}
\item For all stationary subsets $X,S$ of $\kappa$, there exist stationary subsets $X'\s X$ and $S'\s S$
such that $X'$ does not $\mathfrak f$-reflect to $S'$;
\item There exists an injection $h:\mathcal P(\kappa)\rightarrow\ns_\kappa^+$ such that, for all $X,S\in\mathcal P(\kappa)$,
$X\s S$ iff $h(X)$ $\mathfrak f$-reflect to $h(S)$;
\item For all two disjoint stationary subsets $X,S$ of $\kappa$, $X$ does not $\mathfrak f$-reflect to $S$.
\end{enumerate}
\end{cor}
\begin{proof} A moment reflection makes it clear that Clauses (1) and (2) both follow from Clause~(3), so we focus on proving the latter.

Suppose first that $\kappa$ is a successor cardinal, say, $\kappa=\lambda^+$.
Let $\mathbb S$ denote the standard cofinality-preserving notion of forcing for adding a $\square_\lambda$-sequence (see \cite[\S6.1]{MR1838355}).
Now, work in $V[H][G]$, where $H*G$ is  $\mathbb S*\add(\kappa,\kappa^+)$-generic over $V$.
As, in $V[H]$, there exists a coherent regressive $C$-sequence over a club in $\kappa$,
Corollary~\ref{cor5182} entails that, in $V[H][G]$,
for all two disjoint stationary subsets $X,S$ of $\kappa$, $X$ does not $\mathfrak f$-reflect to $S$.

Next, suppose that $\kappa$ is inaccessible.
Let $\mathbb S_0:=(S_0,\le_0)$, where $c\in S_0$ iff $c$ is closed bounded subset of $\kappa$ disjoint from $\reg(\kappa)$,
and $d\le_0 c$ iff $d$ end-extends $c$. It is easy to see that $\mathbb S_0$ is  ${<}\kappa$-distributive,\footnote{For every condition $c$,
and every $\theta\in\reg(\kappa)$, $d:=c\cup\{\sup(c)+\theta\}$ is a condition extending $c$
such that $\mathbb S_0\mathrel{\downarrow}d$ is $\theta^+$-closed.}
of size $\kappa^{<\kappa}=\kappa$,
and shoots a club through $\sing(\kappa)$, so that, $\mathbb S_0$ is cofinality-preserving, and, in $V^{\mathbb S_0}$, $\kappa$ is not Mahlo.
Now, let $\dot{\mathbb S}_1$ be the $\mathbb S_0$-name for the poset from  \cite[\S6]{MR1942302} for adding a coherent regressive $C$-sequence over $\sing(\kappa)$.
Then $\mathbb S:=\mathbb S_0*\dot{\mathbb S}_1$ is a cofinality-preserving notion of forcing that adds a coherent regressive $C$-sequence over a club in $\kappa$.
So, as in the previous case, any forcing extension by $\mathbb S*\add(\kappa,\kappa^+)$ gives the desired model.
\end{proof}

\begin{lemma}\label{IdealProp2}  Let $X\s\kappa$ be stationary, and $\mu\in\reg(\kappa)$.
\begin{enumerate}
\item $\kappa\cap\cof(\omega)\in I[\kappa-X]$;
\item If  $X\s\cof(\ge\mu)$, then $I[\kappa-\emptyset]\restriction\cof(\le\mu)=I[\kappa - X]\restriction\cof(\le\mu)$;
\item If $\mu^{<\mu}<\kappa$, then $S\in I[\kappa-X]\restriction\cof(\le\mu)$ iff
$S\in I[\kappa-\emptyset]\restriction\cof(\le\mu)$ and $\tr(X)\cap S$ is non-stationary.
\item If $2^{<\lambda}=\lambda$ and $X$ is non-reflecting,
then $I[\lambda^+-\emptyset]=I[\lambda^+-X]$.
\end{enumerate}
\end{lemma}
\begin{proof} (1) Let $\langle a_\beta\mid\beta<\kappa\rangle$ be any enumeration of all finite subsets of $\kappa$,
and consider the club $C:=\{\delta<\kappa\mid [\delta]^{<\omega}=\{a_\beta\mid \beta<\delta\}\}$.

(2) The witness is the same.

(3) The forward implication is clear, so we focus on the converse.
Suppose $S\in I[\kappa-\emptyset]\restriction\cof(\le\mu)$ and  $\tr(X)\cap S$ is non-stationary.
Fix a list $\langle a_\beta\mid\beta<\kappa\rangle$ and a club $C$ witnessing together that $S\in I[\kappa-\emptyset]$.
By shrinking $C$, we may assume that $\tr(X)\cap S\cap C=\emptyset$.
As $S\s\cof(\le\mu)$, we may also assume that $|a_\beta|<\mu$ for all $\beta<\kappa$.
Thus, assuming $\mu^{<\mu}<\kappa$,
we may let $\langle a_\beta^\bullet\mid\beta<\kappa\rangle$ be some enumeration of $\bigcup\{\mathcal P(\cl(a_\beta))\mid \beta<\kappa\}$.
Now, define a function $f:\kappa\rightarrow\kappa$ by letting for all $\beta<\kappa$:
$$f(\beta):=\min\{\alpha<\kappa\mid \mathcal P(\cl(a_\beta))\s\{a^\bullet_\gamma\mid \gamma<\alpha\}\}.$$

Put $D:=\{\delta\in C\mid f[\delta]\s\delta\}$.
To see that $\langle a_\beta^\bullet\mid\beta<\kappa\rangle$ and $D$ witness together that $S\in I[\kappa-X]$,
let $\delta\in S\cap D$ be arbitrary. In particular, $\delta\in S\cap C$, and we may fix a cofinal subset $A\s\delta$ of order-type $\cf(\delta)$
such that any proper initial segment of $A$ is listed in $\{ a_\beta\mid\beta<\delta\}$.
As $\delta\in S\cap C$, $\delta\notin\tr(X)$, so we may fix a subclub $A^\bullet$ of $\acc^+(A)$ which is disjoint from $X$.
Now, let $\gamma<\delta$ be arbitrary.
Put $\alpha:=\min(A\setminus\gamma)$ and then find $\beta<\delta$ such that $A\cap\alpha=a_\beta$.
It follows that $A^\bullet\cap\gamma\s\cl(a_\beta)$.
So, as $f[\delta]\s\delta$, we infer that $A^\bullet\cap\gamma$ is indeed listed in $\{a_\beta\mid\beta<\delta\}$.

(4) This follows from (3).
\end{proof}

\begin{cor}\label{cor5192}
Suppose $X\subseteq \kappa$ and $S\subseteq \kappa\cap\cof (\omega)$ are stationary sets,
for which $X\setminus S$ is stationary. After forcing with $\add(\kappa,\kappa^{+})$,
$X$ does not $\mathfrak f$-reflect to $S$.
\end{cor}
\begin{proof} Consider the stationary set $Z:=X\setminus S$.
By Lemma~\ref{IdealProp2}(1), $\kappa\cap\cof(\omega)\in I[\kappa-Z]$, in particular, $S\in I[\kappa-Z]$.
By Corollary~\ref{corollary4122}, we conclude that after forcing with $\add(\kappa,\kappa^{+})$,  $Z$ does not $\mathfrak f$-reflect to $S$.
Now, Monotonicity Lemma~\ref{monotonicity} finishes the proof.
\end{proof}

Thus, we get a slight improvement of Corollary~\ref{PlusNotPlusPlus}:
\begin{cor}\label{cor525} It is consistent that for $S:=\aleph_2\cap\cof(\omega)$ and $X:=\aleph_2\cap\cof(\omega_1)$, $\diamondsuit^+_S$ holds,
but $X$ does not $\mathfrak f$-reflect to $S$.
\end{cor}
\begin{proof}  Suppose $\kappa=\aleph_2=2^{2^{\aleph_0}}$,
and work in the forcing extension by $\add(\kappa,\kappa^{+})$.
By \cite[Lemma~2.1]{MR485361}, $\diamondsuit^+_S$ holds. By Corollary~\ref{cor5192}, $X$ does not $\mathfrak f$-reflect to $S$.
\end{proof}

\section{Dense non-reduction}\label{nonreductionsection}

In the previous section, we showed how adding $\kappa^+$ many Cohen subsets of $\kappa$ could
ensure the failure of instances of $\mathfrak f$-refl at the level of $\kappa$. In this section,
motivated by Lemma~\ref{lemma53}, we shall derive a stronger conclusion.

This section builds heavily on the ideas of Chapter~4, Section~4 of \cite{FHK},
where, given $X,Y$ stationary subsets of $\kappa$, it is forced under certain hypothesis that ${=_{X}^{2}}\nredub{=_{Y}^{2}}$.
Here we adapt their arguments to get ${=_{X}^{2}}\nredum{=_{Y}}$,
let alone ${=_{X}}\nredub{=_{Y}}$ or ${=_{X}^{2}}\nredub{=_{Y}^{2}}$.
In addition, our proof takes advantage of the ideal $I[\kappa-X]$ from the previous section,
hence, the findings here are applicable also for $\kappa$ successor of singular in which $Y$ concentrates on points of cofinality above the cofinality of the singular.

\begin{conv} We denote elements of $\kappa^{<\kappa}$ by English letters (e.g., $p$ and $q$),
and elements of $\kappa^\kappa$ by Greek letters (e.g., $\eta$ and $\xi$).
Subsets of $\kappa$ will be denoted by $X,Y,Z$ and $S$.
For all $\eta\in\kappa^\kappa$ and $A\s\kappa$, we denote $A_\eta:=\{\alpha\in A\mid \eta(\alpha)\neq0\}$.
We also let $\vec 0$ denote the constant $\kappa$-sequence with value $0$.
\end{conv}

\begin{defn}\label{encodingmap} A function $F:2^{<\kappa}\times \kappa^{<\kappa} \rightarrow2$ is said to \emph{encode a map from $2^\kappa$ to $\kappa^\kappa$} iff for all $p\in2^{<\kappa}$ and $q,q'\in\kappa^{<\kappa}$,
$F(p,q)=F(p,q')=1$ entails that $q\cup q'\in \kappa^{<\kappa}$.

The interpretation of $F$ is the function $F^*:2^\kappa\rightarrow\kappa^\kappa$
defined as follows. Given $\eta\in2^\kappa$,
if there exists $\xi\in\kappa^\kappa$ satisfying that, for all $\varepsilon<\kappa$,
there is a tail of $\delta<\kappa$, such that $F(\eta\restriction\delta,\xi\restriction\varepsilon)=1$,
then $\xi$ is unique and we let $F^*(\eta):=\xi$. Otherwise, we let $F^*(\eta):=\vec 0$.
\end{defn}

\begin{defn}\label{boundedtopology} Let $\theta\in[2,\kappa]$. A basic open set in the space $\theta^\kappa$
is a set of the form $N_p:=\{\eta\in\theta^\kappa\mid p\s \eta\}$ for some $p\in\theta^{<\kappa}$.
A binary relation $A\subseteq \theta^\kappa\times\theta^\kappa$ is said to be \emph{analytic} iff there is a closed subset $F$ of the product space $\theta^\kappa\times\theta^\kappa\times\theta^\kappa$ such that $A$ is equal to the projection $\pr(F):=\{(\eta,\xi)\in \theta^\kappa\times\theta^\kappa\mid\exists\zeta\in \theta^\kappa~(\eta,\xi,\zeta)\in F\}$.
\end{defn}

\begin{defn} A subset $D\s2^\kappa$ is said to be \emph{comeager} if
$D\supseteq\bigcap\mathcal D$ for some nonempty family $\mathcal D$ of at most $\kappa$-many dense
open subsets of $2^\kappa$.

A subset of $2^\kappa$ is said to be \emph{meager} iff its complement is {comeager}.
\end{defn}
\begin{fact}[Generalized Baire Category Theorem, {\cite[Theorem~9.87]{Jouko}}]\label{bct}
For any nonempty family $\mathcal D$ of at most $\kappa$-many dense
open subsets of $2^\kappa$, $\bigcap\mathcal D$ is dense.
In particular, the collection $\mathcal M_\kappa$ of all meager subsets of $2^\kappa$ forms a $\kappa^+$-additive proper ideal,
and the intersection of a comeager set with a set of the form $N_p$ (as in Definition~\ref{boundedtopology}) is nonempty.
\end{fact}

\begin{defn} A function $H:\kappa \times \kappa \rightarrow2^{<\kappa}$
is said to \emph{encode a comeager set} iff, for every $i<\kappa$, the fiber $H[\{i\}\times\kappa]$ is cofinal in $(2^{<\kappa},\s)$.

The interpretation of $H$ is the set:$$H^*:= \bigcap_{i < \kappa}\bigcup_{j<\kappa} N_{H(i,j)}.$$
\end{defn}

The following is obvious.
\begin{lemma}
For any comeager set $D\s 2^\kappa$, there is a function $H:\kappa \times \kappa \rightarrow2^{<\kappa}$ encoding a comeager subset $H^*$ of $D$.\qed
\end{lemma}

\begin{defn} A subset $B\s 2^\kappa$ is said to have \emph{the Baire property} iff there is an open set $U\s 2^\kappa$ for which the symmetric difference $U\symdiff B$ is meager.

We denote by $\mathcal B_\kappa$ the collection of all subsets of $2^\kappa$ having the Baire property.
\end{defn}
\begin{prop}\label{bairep} $\mathcal B_\kappa$ is an algebra which is closed under unions of length $\le\kappa$.
\end{prop}
\begin{proof} Given $B\in \mathcal B_\kappa$, to see that $B^c:=2^\kappa\setminus B$ is in $\mathcal B_\kappa$,
fix an open set $U$ such that $U\symdiff B$ is meager, and let $V$ denote the interior of $U^c:=2^\kappa\setminus U$;
then $B^c\symdiff V\s (B^c\symdiff U^c)\cup(U^c\setminus V)=(B\symdiff U)\cup(U^c\setminus V)$ which is the union of two comeager sets.

To see that $\mathcal B_\kappa$ is closed under union of length $\le\kappa$,
note that since $\mathcal M_\kappa$ forms a $\kappa^+$-additive ideal,
this follows from the fact that,
for any sequence of pairs $\langle (B_i,U_i)\mid i<\kappa\rangle$, $\bigcup\{B_i\mid i<\kappa\}\symdiff \bigcup\{U_i\mid i<\kappa\}\subseteq \bigcup\{B_i\symdiff U_i\mid i<\kappa\}$.
\end{proof}

\begin{defn} A function $f:2^{\kappa} \rightarrow\kappa^{\kappa}$ is said to be \emph{Baire measurable} iff for any open set $U\s2^\kappa$, $f^{-1}[U]$ has the Baire property.
The existence of a Baire measurable reduction of $R_0$ to $R_1$ is denoted by $R_0 \redum R_1$.
\end{defn}

\begin{lemma}[folklore]\label{BMcontinuous}
Baire measurable functions are continuous on a comeager set.
\end{lemma}
\begin{proof}
Let $f:2^\kappa\rightarrow\kappa^\kappa$ be a Baire measurable function.
For any $p\in \kappa^{<\kappa}$, choose an open set $U_p\s 2^\kappa$ such that $U_p\symdiff f^{-1}[N_p]$ is meager.
As $\kappa^{<\kappa}=\kappa$, the following set is comeager:
$$D:=\kappa^\kappa\setminus \bigcup_{p\in \kappa^{<\kappa}}(U_p\symdiff f^{-1}[N_p]).$$
Now, to see that $f\restriction D$ is continuous,
it suffices to show that, for every $p\in \kappa^{<\kappa}$, $D\cap f^{-1}[N_p]$ is an open set in $D$. But this is clear, since $D\cap f^{-1}[N_p]=D\cap U_p$.
\end{proof}

\begin{defn}\label{xyCODE} For stationary subsets $X,Y$ of $\kappa$,
we say that $(F,H)$ is an \emph{$(X,Y)$-pair} iff all of the following hold:
\begin{enumerate}
\item $F$ encodes a map from $2^\kappa$ to $\kappa^\kappa$;
\item $H$ encodes a comeager set;
\item For every $\eta\in H^*$,
$$X_\eta\text{ is stationary} \iff  Y_{F^*(\eta)}  \text{ is stationary}.$$
\end{enumerate}
\end{defn}

Our next task is to show that if ${=^2_{X}}\redum{=_{Y}}$, then there exists an $(X,Y)$-pair.
For this, we first introduce the notion of \emph{positivity} of a reduction and prove a lemma about it.

\begin{defn}\label{positivity}
For every function $f:2^{\kappa} \rightarrow \kappa^{\kappa}$,
we define the \emph{positivity of $f$}, $f^{+}:2^{\kappa}\rightarrow \kappa^{\kappa}$, via:

$$f^+(\eta)(\alpha) :=
\begin{cases}
0,& \text{if } f(\eta)(\alpha)=f(\vec 0)(\alpha);\\
f(\eta)(\alpha), & \text{if }  f(\eta)(\alpha) \neq f(\vec 0)(\alpha)  \text{ and }  f(\eta)(\alpha) \neq 0;\\
f(\vec 0)(\alpha), & \text{if } f(\eta)(\alpha) \neq f(\vec 0)(\alpha)  \text{ and }   f(\eta)(\alpha)=0.
\end{cases} $$
\end{defn}

\begin{lemma}\label{lemma512} Suppose $f:2^{\kappa} \rightarrow \kappa^{\kappa}$ is a Baire measurable function.
Then so is $f^{+}$.
\end{lemma}
\begin{proof} Since every open set is the union of at most $\kappa^{<\kappa}=\kappa$ many sets of the form $N_p$, and since $\mathcal B_\kappa$ is closed under unions of length $\le\kappa$,
it suffices to verify that for every $p \in\kappa^{<\kappa}$, $(f^{+})^{-1}[N_{p}]$ has the Baire property.
Given $p\in\kappa^{<\kappa}$, if there are no $\eta \in 2^\kappa$ and $\beta<\kappa$ such that $ p=f^{+}(\eta)\restriction\beta $, then $(f^{+})^{-1}[N_{p}]$ is empty, and we are done.
Thus, it suffices to consider $p$'s of the form $f^{+}(\eta)\restriction\beta$ for $\eta \in2^\kappa$ and $\beta<\kappa$.

Let $\eta \in2^\kappa$ and $\beta<\kappa$ be arbitrary. Denote $\xi:=f^{+}(\eta)$.
In order to compute $(f^{+})^{-1}[N_{\xi \restriction \beta}]$, we consider the following sets defined to handle each case in the definition of $f^+$:
\begin{itemize}
\item[a)] $B_{0} :=\{\alpha < \beta\mid f(\eta)(\alpha)=f(\vec 0)(\alpha) \}$;
\item[b)] $B_{1} := \{\alpha < \beta\mid f(\eta)(\alpha) \neq f(\vec 0)(\alpha)\text{ and }f(\eta)(\alpha) \neq 0\}$;
\item[c)] $B_{2}:=\{\alpha <\beta\mid f(\eta)(\alpha)\neq f(\vec 0)(\alpha)\text{ and }f(\eta)(\alpha)=0\}$.
\end{itemize}

Next, consider the following sets:
\begin{itemize}
\item[A)] $Q_{0} :=\{q \in \kappa^{\beta} \mid q\restriction B_0 = f(\vec 0)\restriction B_0 \}$;
\item[B)] $Q_{1} :=\{q \in \kappa^{\beta} \mid q\restriction B_1=f(\eta)\restriction B_1 \}$;
\item[C)] $Q_{2} :=\{q \in \kappa^{\beta} \mid q\restriction B_2 = f(\eta) \restriction B_2 \}$.
\end{itemize}

As $f$ is Baire measurable, for each $i<3$,
$$W_{i} := \bigcup\{ f^{-1}[N_{q}] \mid q \in Q_{i} \}$$
is the union of $\le\kappa^{<\kappa}=\kappa$ many sets, each having the Baire property.
So, by Proposition~\ref{bairep}, $W_0 \cap W_1 \cap W_2$ has the Baire property.
Thus, the next claim finishes the proof.
\begin{claim} $(f^{+})^{-1}[N_{\xi\restriction \beta}]= W_{0}\cap W_{1} \cap W_{2}$.
\end{claim}
\begin{proof}
$(\implies)$ Suppose $\zeta \in W_0 \cap W_1 \cap W_2$.
Then $f(\zeta) \restriction B_{0} = f(\vec 0) \restriction B_{0} = f(\eta)\restriction B_0$, $f(\zeta)\restriction B_{1} = f(\eta)\restriction B_{1}$ and $f(\zeta)\restriction B_{2}= f(\eta)\restriction B_{2}$, which implies that $f^{+}(\zeta) \in N_{\xi\restriction\beta}$.

$(\impliedby)$ Suppose $\zeta\in2^\kappa$ with $f^+(\zeta)\in N_{\xi\restriction \beta}$. Then
$$f^{+}(\zeta)\restriction \beta = \xi\restriction \beta = f^{+}(\eta)\restriction \beta = (\vec 0 \restriction B_{0}) \cup (f(\eta)\restriction B_1 )\cup (f(\vec 0)\restriction B_{2}).$$
This implies that
\begin{itemize}
\item $f(\zeta) \restriction B_{0} = f(\vec 0)\restriction B_{0}$,
\item $f(\zeta)\restriction B_{1} = f(\eta) \restriction B_{1} $, and
\item $f(\zeta)\restriction B_{2} = f(\eta) \restriction B_{2} = \vec 0 \restriction B_{2}$.
\end{itemize}

Thus $f(\zeta) \in \bigcap_{i<3}\bigcup_{q \in Q_i}N_q$, so that $\zeta \in W_0 \cap W_1 \cap W_2$.
\end{proof}
This completes the proof.
\end{proof}

\begin{lemma}\label{derivedReduction}
If ${=^2_{X}}\redum{=_{Y}}$, then there exists an $(X,Y)$-pair.
\end{lemma}
\begin{proof}
Suppose $f:2^{\kappa} \rightarrow \kappa^{\kappa}$ is a Baire measurable reduction from $=^2_{X} $ into $=_{Y}$.
Let $f^+$ be the positivity of $f$, so that, by Lemma~\ref{lemma512}, $f^+$ is Baire measurable.

By Lemma~\ref{BMcontinuous},    Baire measurable functions are continuous on a comeager set,
so we may  choose a function $H:\kappa\times\kappa\rightarrow 2^{<\kappa}$
encoding a comeager set $H^*$ and satisfying that $f^{+}\restriction H^*$ is continuous.
Now, define a function $F:2^{<\kappa}\times \kappa^{<\kappa} \rightarrow 2$ via $F(p,q):=1$ iff
$f^+[N_{p} \cap H^*]\s N_{q}$.
\begin{claim}  $(F,H)$ is an  $(X,Y)$-pair.
\end{claim}
\begin{proof} We go over the clauses of Definition~\ref{xyCODE}:
\begin{enumerate}
\item Suppose $p\in 2^{<\kappa}$, $q,q'\in \kappa^{<\kappa}$ and $F(p,q)=1=F(p,q')$.
Then $f^{+}[N_{p}\cap H^*] \subseteq N_{q}$ and $f^{+}[N_{p} \cap H^* ] \subseteq N_{q'}$.
Hence $f^+[N_p\cap H^*]\subseteq N_q\cap N_{q'}$. By Remark~\ref{bct}, the former is nonempty,
so the latter is nonempty and $q\cup q'\in \kappa^{<\kappa}$.
Altogether, $F$ encodes a map from $2^{\kappa}$ to $\kappa^{\kappa}$.

Let us show that $F^*\restriction H^*=f^+\restriction H^*$. For this, let $\eta \in H^* $ be arbitrary.
Set $\xi:=f^+(\eta)$. By the continuity of $f^+\restriction H^*$, for every $\varepsilon<\kappa$, there is a tail of $\delta < \kappa$ such that
$$ N_{\eta \restriction \delta} \cap H^*\s (f^+) ^{-1}[N_{\xi \restriction \varepsilon}],$$
so $ f^+[N_{\eta \restriction \delta}\cap H^*]\s N_{\xi \restriction \varepsilon}$.
Thus $F(\eta \restriction \delta, \xi \restriction \varepsilon) = 1$ for a tail of $\delta<\kappa$,
and it follows from Definition~\ref{encodingmap} that $F^{*}(\eta)= \xi$.

\item By its very choice, $H$ encodes the comeager set $H^*$.
\item Since $F^*\restriction H^*=f^+\restriction H^*$,
we fix an arbitrary $\eta\in H^*$ and prove that  $X_\eta$ is stationary iff $Y_{f^{+}(\eta)}$ is stationary.
The following are equivalent:
\begin{enumerate}
\item $X_\eta$ is stationary;
\item $\{ \alpha \in X \mid \eta(\alpha) \neq \vec{0}(\alpha) \}$ is stationary;
\item $\neg(\eta =^2_{X} \vec{0})$;
\item $\neg(f(\eta) =_{Y} f(\vec{0}))$;
\item $\{ \alpha \in Y \mid f(\eta)(\alpha) \neq f(\vec{0})(\alpha) \}$ is stationary;
\item $\{\alpha \in Y \mid f^+(\eta)(\alpha) \neq 0\}$ is stationary.
\item $Y_{f^{+}(\eta)}$ is stationary.
\end{enumerate}

The equivalence of (c) and (d) follows from the fact that $f$ reduces ${=^2_{X}}$ to ${=_{Y}}$.
The equivalence of (e) and (f) follows from Definition~\ref{positivity}\qedhere
\end{enumerate}
\end{proof}

This completes the proof.
\end{proof}

\begin{thm}\label{CohenEffect} Suppose $X,Y$ are disjoint stationary subsets of $\kappa$,
with $Y \in\allowbreak{I[\kappa - X]}$.
For every pair $(F,H)$, $V^{\add(\kappa,1)}\models (F,H)\text{ is not an }(X,Y)\text{-pair}$.
\end{thm}
\begin{proof} Towards a contradiction, suppose that $(F,H)$ is a counterexample.
As $\add(\kappa,1)$ is almost homogeneous and $X,Y,F,H$ live in the ground model,
it follows that, in fact, $V^{\add(\kappa,1)}\models (F,H)\text{ is an }(X,Y)\text{-pair}$.

Let $\mathbb R, \mathbb P, G , G_0, \mathbb Q$, and $G_1$ be all defined as in the proof of Theorem~\ref{CohenEffect2}.
Denote $\eta:=\bigcup G_0$.

\begin{claim}\label{Claim4112} In $V[G_0]$, $\eta\in H^*$.
\end{claim}
\begin{proof} Let $i<\kappa$.
The set $H[\{i\}\times\kappa]$ is cofinal in $(2^{<\kappa},\s)$,
so, in particular, $D_i:=\{ p{}^\curvearrowright 0\mid p\in H[\{i\}\times\kappa]\}$ is dense in $\mathbb P$.
Pick $p\in G_0\cap D_i$. Then $\eta\in N_p\s N_{p\restriction\max(\dom(p))}$
and the latter is equal to $N_{H(i,j)}$ for some $j<\kappa$.
It thus follows that  $\eta\in\bigcap_{i<\kappa}\bigcup_{j<\kappa}N_{H(i,j)}= H^*$.
\end{proof}

\begin{claim}\label{claim418} In $V[G_0]$, $Y_{F^*(\eta)}$ is stationary.
\end{claim}
\begin{proof}
In $V[G_0]$, $(F,H)\text{ is an }(X,Y)\text{-pair}$, thus, to prove that $Y_{F^*(\eta)}$ is stationary,
it suffices to prove that $X_\eta$ is stationary.
But the latter is precisely what is proved in Claim~\ref{cohen2}.
\end{proof}

\begin{claim}\label{Claim4114} In $V[G_0][G_1]$, $Y_{F^*(\eta)}$ is stationary.
\end{claim}
\begin{proof} This is the content of Claim~\ref{cohen3}.
\end{proof}
\begin{claim}\label{Claim4115} In $V[G_0][G_1]$, $X_{\eta}$ is non-stationary.
\end{claim}
\begin{proof} This is the content of Claim~\ref{cohen4}.
\end{proof}
By Claims \ref{Claim4112}, \ref{Claim4114} and \ref{Claim4115} we infer that, in $V[G]$, $(F,H)$ is not an $(X,Y)$-pair,
contradicting the choice of $(F,H)$.
\end{proof}
\begin{cor}\label{corollary412} Suppose $X,Y$ are disjoint stationary subsets of $\kappa$,
with $Y \in I[\kappa - X]$.
After forcing with $\add(\kappa,\kappa^{+})$, ${=^2_X}\nredum{=_Y}$.
\end{cor}
\begin{proof} Let $G$ be $\add(\kappa,\kappa^+)$-generic over $V$.
For every $\iota\le\kappa^+$, let $G_\iota$ denote the projection of $G$ into the $\iota^{th}$ stage.

Work in $V[G_{\kappa^+}]$.
Towards a contradiction, suppose that $f:2^\kappa\rightarrow\kappa^\kappa$ is a Baire measurable reduction from $=^2_X$ to $=_Y$.
By Lemma~\ref{derivedReduction}, we may fix an $(X,Y)$-pair, say $(F,H)$.
As $F$ and $H$ are elements of $H_{\kappa^+}$
and as $\add(\kappa,\kappa^+)$ has the $\kappa^{+}$-cc, we may find a large enough $\iota<\kappa^+$
such that $(F,H)$ is in $V[G_\iota]$.
Now, by Theorem~\ref{CohenEffect},
$V[G_{\iota+1}]\models`` (F,H)\text{ is not an }(X,Y)\text{-pair}"$.
As $V$ and $V[G_{\kappa^+}]$ have the same bounded subsets of $\kappa$,
the fact that $(F,H)$ is an $(X,Y)$-pair in $V[G_{\kappa^+}]$,
but not in $V[G_{\iota+1}]$ must mean that Clause~(3) of Definition~\ref{xyCODE} fails in $V[G_{\iota+1}]$.
It thus follows that there exists a stationary subset of $\kappa$ in $V[G_{\iota+1}]$ that ceases to be stationary in $V[G_{\kappa^+}]$.
However, the quotient forcing $\add(\kappa,\kappa^+)/G_{\iota+1}$ is isomorphic to $\add(\kappa,\kappa^+)$
and the latter preserves stationary subsets of $\kappa$. This is a contradiction.
\end{proof}

\begin{cor}\label{cor518}  Suppose that there exists a coherent regressive
$C$-sequence over a stationary subset $\Gamma$ of $\kappa$.
After forcing with $\add(\kappa,\kappa^{+})$,
for all two disjoint stationary subsets $X,Y$ of $\Gamma$,
${=^2_X}\nredum{=_Y}$.
\end{cor}
\begin{proof}
Let $G$ be  $\add(\kappa,\kappa^+)$-generic over $V$.
As in the proof of Corollary~\ref{corollary412},
for every $\iota\le\kappa^+$, we let $G_\iota$ denote the projection of $G$ into the $\iota^{th}$ stage.

Work in $V[G]$. Suppose that $X$ and $Y$ are disjoint stationary subsets of $\Gamma$.
As $X$ and $Y$ are elements of $H_{\kappa^+}$
and as $\add(\kappa,\kappa^+)$ has the $\kappa^{+}$-cc, we may find a large enough $\iota<\kappa^+$
such that $X$ and $Y$ are in $V[G_\iota]$.
Now, in $V$, and hence also in $V[G_\iota]$, there exists a coherent regressive $C$-sequence over $\Gamma$.
So, by Lemma~\ref{usingsquare}, there is a stationary $Z\subseteq X$ such that $Y\in I[\kappa-Z]$.
Finally, since $\add(\kappa,\kappa^+)/G_{\iota}$ is isomorphic to $\add(\kappa,\kappa^+)$,
Corollary~\ref{corollary412} implies that, in $V[G]$, ${=_{Z}^2}\nredum{=_Y}$.
As $Z\s X$, it follows from Monotonicity Lemma~\ref{transitivity} that,  in $V[G]$, ${=_X^2}\nredum{=_Y}$.
\end{proof}

\begin{cor}[Dense non-reduction]\label{DNR} There exists a cofinality-preserving forcing extension
in which:
\begin{enumerate}
\item For all stationary subsets $X,Y$ of $\kappa$, there exist stationary subsets $X'\s X$ and $Y'\s Y$ such that ${=^2_{X'}}\nredum{=_{Y'}}$;
\item There exists an injection $h:\mathcal P(\kappa)\rightarrow\ns_\kappa^+$ such that, for all $X,Y\in\mathcal P(\kappa)$,
$X\s Y$ iff ${=^2_{h(X)}}\redum{=_{h(Y)}}$;
\item For all two disjoint stationary subsets $X,Y$ of $\kappa$, ${=^2_X}\nredum{=_Y}$.
\end{enumerate}
\end{cor}
\begin{proof} This is the same model of Corollary~\ref{corollary514}. That is,
we force to add a coherent regressive $C$-sequence over a club in $\kappa$,
and then add $\kappa^+$ many Cohen subsets of $\kappa$.
The only difference is that this time we appeal to Corollary~\ref{cor518} instead of to Corollary~\ref{cor5182}.
\end{proof}

\begin{cor}\label{cor519}
Suppose $X\subseteq \kappa$ and $S\subseteq \kappa\cap\cof (\omega)$ are stationary sets,
for which $X\setminus S$ is stationary. After forcing with $\add(\kappa,\kappa^{+})$,  ${=^2_X}\nredum{=_S}$.
\end{cor}
\begin{proof} Consider the stationary set $Z:=X\setminus S$.
By Lemma~\ref{IdealProp2}(1), $\kappa\cap\cof(\omega)\in I[\kappa-Z]$, in particular, $S\in I[\kappa-Z]$.
By Corollary~\ref{corollary412}, we conclude that after forcing with $\add(\kappa,\kappa^{+})$, ${=^2_Z}\nredum{=_S}$.
Now, Monotonicity Lemma~\ref{transitivity} finishes the proof.
\end{proof}

Based on the work done thus far, we are now able to answer a few questions from the literature.
\begin{defn}An equivalence relation $R$ is said to be \emph{$\Sigma^1_1$-complete} iff it is analytic and, for every analytic equivalence relation $E$, ${E}\redub{R}$.
\end{defn}

\begin{question}[Aspero-Hyttinen-Kulikov-Moreno, {\cite[Question 4.3]{AHKM}}]
Is it consistent that $\kappa$ is inaccessible and $=^2_S$ is not $\Sigma_1^1$-complete for some stationary $S\s\kappa$?
\end{question}

We answer the preceding in the affirmative:

\begin{thm} If $\kappa$ is an inaccessible cardinal, then there exists a cofinality-preserving forcing extension in which ($\kappa$ is inaccessible, and) for every stationary co-stationary $S\s \kappa$, $=^2_S$ is not a $\Sigma^1_1$-complete equivalence relation.
\end{thm}
\begin{proof} This is the forcing extension of Corollary~\ref{DNR}.
In this model, for all two disjoint stationary subsets $X,S$ of $\kappa$,
${=^2_X}\nredum{=_S}$.
In particular ${=^2_X}\nredum{=^2_S}$, so that  $=^2_S$ is not $\Sigma_1^1$-complete.
\end{proof}

\begin{question}[Moreno, {\cite[Question 4.16]{Mor17}}]
Is it consistent that $${=_{\kappa\cap\cof (\mu)}}\nredub{=_{\kappa\cap\cof (\nu)}}$$
holds for all infinite regular cardinals $\mu\neq\nu$ below $\kappa$?
\end{question}

We answer the preceding in the affirmative:

\begin{thm}
There is a cofinality-preserving forcing extension, in which,
for all infinite regular cardinals $\mu\neq\nu$ below $\kappa$, $=_{\kappa\cap\cof (\mu)}\nredum =_{\kappa\cap\cof (\nu)}$.
\end{thm}
\begin{proof} This is the forcing extension of Corollary~\ref{DNR}.
For all infinite regular cardinals $\mu\neq\nu$ below $\kappa$,
$\kappa\cap\cof(\mu)$ and $\kappa\cap\cof(\nu)$ are disjoint stationary subsets of $\kappa$,
and hence ${=^2_{\kappa\cap\cof (\mu)}}\nredum{=_{\kappa\cap\cof (\nu)}}$.
In particular, ${=_{\kappa\cap\cof (\mu)}}\nredum{=_{\kappa\cap\cof (\nu)}}$.
\end{proof}

\begin{question}[Aspero-Hyttinen-Kulikov-Moreno, {\cite[Question 2.12]{AHKM}}]
Is it consistent that, for all infinite regular $\mu<\nu<\kappa$,
the following hold?
$${=_{\kappa\cap\cof (\mu)}}\redub{=^2_{\kappa\cap\cof (\nu)}}\ \&\ {=^2_{\kappa\cap\cof (\nu)}}\nredub{=_{\kappa\cap\cof (\mu)}}.$$
\end{question}

We answer the preceding in the affirmative, along the way, proving Theorem~E:

\begin{thm}\label{MM-application}
Suppose $\mm$ holds. After forcing with $\add(\omega_2,\omega_3)$, $\mm$ still holds,
and so are all of the following:
\begin{enumerate}
\item ${=_{S^2_0}}\reduO{=^2_{S^2_1}}$;
\item For every stationary $X\subseteq S^2_1$, ${=^2_X}\nredum{=_{S^2_0}}$;
\item There are stationary subsets $X\subseteq S^2_0$ and $Y\s S^2_1$ such that ${=^2_X}\nredum{=_Y}$;
\item There is a stationary $Y\s S^2_1$ such that ${=^2_{S^2_1}}\nredum{=_Y}$;
\item ${=_{S^2_0}}\reduO{=^2_{S^2_1}}$ and ${=^2_{S^2_1}}\nredum{=_{S^2_0}}$.
\end{enumerate}

\end{thm}
\begin{proof} We start with a model of $V\models\mm$,
and pass to $V[G]$, where $G$ is $\add(\omega_2,\omega_3)$-generic over $V$.
By \cite[Theorem~4.3]{MR1782117}, $\mm$ is preserved by ${<}\omega_2$-directed-closed forcing,
so that $V[G]\models\mm$.

(1) $\mm$ implies that $2^{\aleph_1}=\aleph_2$, which, by a Theorem of Shelah \cite{Sh:922}, implies that $\diamondsuit_{S^2_0}$ holds.
Now appeal to Lemma~\ref{MM}.

(2) This is Corollary~\ref{cor519}.

(3) Work in $V$. By \cite{Sakaisquare}, $\mm$ implies the existence of a coherent regressive $C$-sequence over some $\Gamma\s\omega_2$ for which $Y:=\Gamma\cap\cof(\omega_1)$ is stationary.
A moment reflection makes it clear that $X:=\Gamma\cap\cof(\omega)$ must be stationary, as well. Now, by Corollary~\ref{cor518},
in $V[G]$, ${=^2_X}\nredum{=_Y}$.

(4) By Clauses (1) and (3).

(5) By Clauses (1) and (2).
\end{proof}

\section*{Acknowledgements}
This research was partially supported by the European Research Council (grant agreement ERC-2018-StG 802756). The third author was also partially supported by the Israel Science Foundation (grant agreement 2066/18).
At the end of the preparation of this article, the second author was visiting the Kurt G\"odel Research Center supported by the Vilho, Yrj\"o and Kalle V\"ais\"al\"a Foundation of the Finnish Academy of Science and Letters.

We thank Andreas Lietz, Benjamin Miller, Ralf Schindler and Liuzhen Wu for their feedback on a preliminary version of this manuscript.

\end{document}